\documentclass[10pt,reqno]{amsart}
\usepackage{latexsym}
\usepackage{graphicx}
\usepackage{amssymb}

\usepackage{amsmath}
\usepackage{amssymb}
\usepackage{tikz}
\usepackage{pgfplots}
\pgfplotsset{compat=1.18}
\usepackage{physics}
\usepackage{tikz-cd}
\usepackage[english]{babel}
\usepackage{amsthm}
\usepackage[hidelinks]{hyperref}
\usepackage{cleveref}
\usepackage{graphicx}
\usepackage{subcaption}
\usepackage{booktabs}
\usepackage{tabularx}
\usepackage{aliascnt}

\newcommand{\Lie}{\operatorname{Lie}}

\newcommand{\GL}{\operatorname{GL}}

\renewcommand{\rank}{\operatorname{rank}}
\newcommand{\mcA}{\mathcal{A}}
\newcommand{\mcG}{\mathcal{G}}
\DeclareMathOperator{\maxsing}{maxsing}

\DeclareMathOperator{\pt}{pt}
\DeclareMathOperator{\fl}{fl}
\DeclareMathOperator{\Br}{Br}

\newcommand{\bigcircledbullet}{%
  \mathbin{\tikz[baseline=-0.6ex]{
    \draw (0,0) circle (0.6ex);
    \fill (0,0) circle (0.25ex);
  }}%
}

\theoremstyle{plain}
\newtheorem{theorem}{Theorem}[section]

\newaliascnt{lemma}{theorem}
\newtheorem{lemma}[lemma]{Lemma}
\aliascntresetthe{lemma}

\newaliascnt{proposition}{theorem}
\newtheorem{proposition}[proposition]{Proposition}
\aliascntresetthe{proposition}

\newaliascnt{corollary}{theorem}
\newtheorem{corollary}[corollary]{Corollary}
\aliascntresetthe{corollary}

\newaliascnt{conjecture}{theorem}
\newtheorem{conjecture}[conjecture]{Conjecture}
\aliascntresetthe{conjecture}

\newaliascnt{claim}{theorem}

\aliascntresetthe{claim}

\theoremstyle{definition}

\newaliascnt{definition}{theorem}
\newtheorem{definition}[definition]{Definition}
\aliascntresetthe{definition}

\theoremstyle{remark}

\newaliascnt{remark}{theorem}
\newtheorem{remark}[remark]{Remark}
\aliascntresetthe{remark}

\makeatletter
\newcommand{\addresseshere}{%
  \enddoc@text\let\enddoc@text\relax
}
\makeatother

\begin{document}

\setcounter{page}{1}

\title[On Open Borel Orbits in Conormal Bundles to Schubert Cells]{On Open Borel Orbits in Conormal Bundles to Schubert Cells}
\author{Derek R. Aoki}
\address{Proof School\\
                221 Main St.\\
                San Francisco, California\\
                94103\\
                USA}
\email{derekaoki@acm.org}

\begin{abstract}
    Let $C(v)$ denote a Schubert cell in the flag variety $\GL_n/B$, and let $X(w)$ be a Schubert variety containing $C(v)$. Mellit has conjectured that, if $X(w)$ is smooth, the existence of an open $B$-orbit in the conormal bundle to $C(v)$ in $X(w)$ is equivalent to $C(w_0w)$ being contained in the smooth locus of $X(w_0v)$. We prove a stronger version of this conjecture, requiring only that $X(w)$ is factorial. We prove our result by constructing a sequence of graphs associated to the pair $v,w$ such that the existence of a cycle in one of these graphs is equivalent to the non-existence of such an open orbit.
\end{abstract}

\maketitle

\noindent\par
\medskip

\indent\indent {\fontsize{8pt}{9pt} \selectfont MSC2020: 14M15, 05E14, 05C38 \par}
\indent\indent {\fontsize{8pt}{9pt} \selectfont Key Words and Phrases: Schubert varieties, conormal bundles, Bruhat order, singular loci \par}


\section{Introduction}

Let $G = \GL_n$ and let $B \subset G$ be its Borel subgroup of upper triangular matrices. The flag variety $G/B$ is stratified by Schubert varieties $X(w)$ defined as the closures of Schubert cells $C(w)$, the orbits of $B$ on $G/B$ acting by left multiplication. The geometry of this stratification has proven fundamental to both representation theory and algebraic combinatorics. Kazhdan and Lusztig conjectured \cite{KazhdanLusztig1979,KazhdanLusztig1980} that certain polynomials, indexed by pairs $v,w\in S_n$ such that $v\preceq w$, control the combinatorics of the BGG category $\mathcal{O}$ for $\mathrm{Lie}(G)$. These conjectures were later proven independently by Brylinski--Kashiwara \cite{BrylinskiKashiwara1981}  and Beilinson--Bernstein \cite{BeilinsonBernstein1981}, who both tied the Kazhdan--Lusztig polynomials $P_{v,w}(q)$ to the geometry of Schubert varieties. For example, the (rational) smoothness of $C(v)$ in $X(w)$ is equivalent to the triviality of $P_{v,w}(q)$ \cite{Deodhar1985}. 

Motivated by these connections as well as the combinatorial properties of flag varieties, the last few decades have seen interest in understanding the smooth locus of the Schubert variety $X(w)$. The seminal result on this topic by Lakshmibai--Sandhya \cite{LakshmibaiSandhya1990} showed that $X(w)$ is smooth if and only if $w$ avoids 3412 and 4231. Following their result, four independently published papers gave a combinatorial characterization of the pairs $v,w\in S_n$ such that $C(v)$ is contained in the singular locus of $X(w)$ \cite{billey2001maximalsingularlocischubert,Cortez2003,KasselLascouxReutenauer2003,Manivel2001}. Lapid--M\'inguez showed (albeit in the equivalent language of $G$-orbits in $G/B\times G/B$) using combinatorial methods that, if $w$ avoids 231, then an open $B$-orbit in the conormal bundle to $C(v)\subseteq X(w)$ exists if and only if $C(w_0w)$ is contained in the smooth locus of $X(w_0v)$ \cite{Lapid_2018}. This open orbit condition originates with the work of Gei\ss-Leclerc-Schr\"oer \cite{GeissLeclercSchroer2005,GeissLeclercSchroer2006,GeissLeclercSchroer2011} and subsequent work on rigid modules over preprojective algebras and related dense orbit conditions \cite{AizenbudLapid2025,GeissSchroer2005,JensenSuYu2009,Lapid2026,LapidMinguez2020}. This result of Lapid--M\'inguez led Mellit to conjecture that this equivalence in fact holds when $X(w)$ is smooth, which is a strictly weaker condition than $w$ avoiding 231 \cite[Conjecture 6.2]{LapidMinguez2025}. In this paper, we study the following equivalent formulation of Mellit's conjecture.

\begin{conjecture}
    Let $v\preceq w$ such that $X(w)$ is smooth. Then $X(w_0v)$ is singular at $e_{w_0w}$ if and only if there is no open $B$-orbit in the conormal bundle to $C(v)\subseteq X(w)$.
\end{conjecture}

Our main contribution is the following theorem, which proves a stronger version of the conjecture above.

\begin{theorem}\label{thm:maintheorem}
    Let $v\preceq w$ such that $X(w)$ is factorial. Then $X(w_0v)$ is singular at $e_{w_0w}$ if and only if there is no open $B$-orbit in the conormal bundle to $C(v)\subseteq X(w)$.
\end{theorem}

To prove our main theorem, we establish a combinatorial condition equivalent to the existence of an open $B$-orbit in the conormal bundle. We construct a sequence of graphs $\mcG_i(v,w)$ that we associate to the pair $v,w$, which we build from the pairs $(a,b)$ such that $v\prec vs_{ab}\preceq w$. A construction equivalent to $\mcG_1(v,w)$ had previously been considered by Tsukerman--Williams in \cite{TsukermanWilliamsBruhatIntervalPolytopes}, called the atom graph. On the other hand, to the best of our knowledge, the higher graphs $\mcG_i(v,w)$ for $i>1$ have not been studied. We will show that the presence of a cycle in these graphs is equivalent to the failure of the conormal bundle to contain an open $B$-orbit. Our approach makes heavy use of the characterization of maximal singular points by Billey--Warrington \cite{billey2001maximalsingularlocischubert} and their corresponding Bruhat pictures.

In \Cref{sec:preliminaries}, we recall a linear algebraic characterization of the existence of an open $B$-orbit from \cite{Lapid_2018} and a pictorial characterization of the maximal singular points from \cite{billey2001maximalsingularlocischubert} in terms of Bruhat pictures. In \Cref{sec:maxsing}, we prove that the existence of an open $B$-orbit implies that $C(w_0w)$ is contained in the smooth locus of $X(w_0v)$, without any assumptions on $w$. In \Cref{sec:necessarycondition}, we define the graphs $\mcG_i(v,w)$ and use them to give a necessary condition for the non-existence of an open $B$-orbit. 

In \Cref{sec:maintheorem}, we prove our main theorem. We first prove the following theorem, which is of independent interest because it gives a purely combinatorial characterization of the existence of an open $B$-orbit.

\begin{theorem}
    There exists an open $B$-orbit in the conormal bundle to $C(v)\subseteq X(w)$ if and only if each of the graphs $\mcG_i(v,w)$ is a forest.
\end{theorem}

In fact, an algorithm utilizing the above theorem for determining the existence of an open $B$-orbit runs in $O(n^2)$, whereas the naive algorithm using the linear algebraic characterization of \cite{Lapid_2018} runs in $O(n^6)$. Further, we have that a cycle in $\mcG_i(v,w)$ for $i>1$ implies $X(w_0v)$ is singular at $e_{w_0w}$. We conclude by proving that if $X(w)$ is factorial and $\mcG_1(v,w)$ contains a cycle then $X(w_0v)$ is singular at $e_{w_0w}$, using results of Deodhar \cite{Deodhar1985} and Dyer \cite{Dyer1997} on the nontriviality of Kazhdan-Lusztig polynomials and results of Bousquet--M\'elou--Butler \cite{BousquetMelouButler2007} and Tsukerman--Williams \cite{TsukermanWilliamsBruhatIntervalPolytopes} on the graphs $\mcG_1(v,w)$.

\begin{remark}
    During the preparation of this paper, the preprint \cite{Lapid2026} by Lapid appeared. In its treatment of type $A$, it makes use of the same linear algebraic condition from \cite{Lapid_2018} we do, and defines a graph which is the same as $\mcG_1(v,w)$ up to choice of conventions. This paper also characterizes the existence of an open $B$-orbit in terms of the termination of a greedy algorithm. Our approach gives a more explicit combinatorial description in terms of the graphs $\mcG_i(v,w)$. Moreover, our paper handles the case of arbitrary pairs $v\preceq w$, whereas \cite{Lapid2026} only handles the case $w=w_0$. Because of this, we are able to prove Mellit's conjecture, whereas the goals of \cite{Lapid2026} are different.
\end{remark}

\section{Preliminaries}\label{sec:preliminaries}

\subsection{Basic Definitions}

Throughout this paper, we work over $\mathbb C$. Again, let $G=\GL_n$, and let $B$ be the Borel subgroup of upper triangular elements of $G$. We write $T \subset B$ for the maximal torus of diagonal matrices. The flag variety $G/B$ is a disjoint union of $B$-orbits $C(w)=BwB/B$ in $G/B$ called \textit{Schubert cells}, which are indexed by elements $w$ of the Weyl group $W=N(T)/T$. We have an isomorphism $W\cong S_n$, where $S_n$ is the symmetric group on $n$ elements. We write $s_{ij}$ for the element of $W\cong S_n$ transposing $i$ and $j$. We call the Zariski closure $X(w)=\overline{C(w)}$ a \textit{Schubert variety}. Since each $X(w)$ is stable under the action of $B$, it is a union of Schubert cells. We define a partial order, called the \textit{Bruhat order}, on $W\cong S_n$ such that $v\preceq w$ if and only if $C(v)\subseteq X(w)$. There exists a unique element denoted $w_0\in W$, defined by $i\mapsto n+1-i$, which is maximal with respect to the Bruhat order. In the sequel, we refer to permutations by their one line notations. That is, we write $w=[a_1,\dots,a_n]$ for the permutation $w$ such that $w(i)=a_i$. For instance, we have $w_0=[n,n-1,\dots,1]$.

For $v \in S_n$, we write $e_v$ for the point $vB/B\in G/B$. Writing $B^v = vBv^{-1}$, we define $B_v = B \cap B^v$ for the stabilizer of $e_v \in G/B$ under the $B$-action. We also write $\mathfrak b^v=\Lie(B^v)$ and $\mathfrak b_v=\Lie(B_v)$. We identify the tangent space $T_{e_v}(G/B)$ to $G/B$ at $e_v$ with the quotient $\mathfrak g/\mathfrak b^v$. Writing $B^-$ for the Borel subgroup opposite to $B$, and writing $\mathfrak n_v^-$ for the nilpotent radical of $\Lie((vB^-v^{-1})$, we identify $\mathfrak g/\mathfrak b^v$ with $\mathfrak n_v^-$ as vector spaces. The cotangent space $T^\vee_{e_v}(G/B)$ is the linear dual to $T_{e_v}(G/B)$, and $B_v$ acts naturally on $T_{e_v}(G/B)$. The nondegenerate bilinear form $(X,Y)\mapsto \operatorname{tr}(XY)$ then yields a $B_v$-equivariant isomorphism $(\mathfrak g/\mathfrak b^v)^*\cong \mathfrak n_v^+$, the nilpotent radical of $\mathfrak b^v$, where $B_v$ acts on $\mathfrak n_v^+$ via the adjoint action. Explicitly, we fix this identification so that $T_{e_v}^\vee (G/B)$ is spanned by the elementary matrices $E_{v(i)v(j)}$ for $i<j$. 

Let $w\in S_n$ such that $w\succeq v$. The following theorem of Lakshmibai--Seshadri gives an explicit description of the tangent spaces to the Schubert varieties as a subspace of $T_{e_v}(G/B)$ {\cite[Theorem 1]{LakshmibaiSeshadri1984}}.

\begin{theorem}
    The tangent space $T_{e_v}X(w)\subseteq T_{e_v}(G/B)$ has a basis $\{E_{v(j)v(i)}:i< j,vs_{ij}\preceq w\}$, where $s_{ij}$ denotes the permutation transposing $i$ and $j$.  
\end{theorem}

Dualizing, this identification yields a $B_v$-equivariant surjection $T_{e_v}^\vee (G/B)\twoheadrightarrow T_{e_v}^\vee X(w)$, with kernel $K(v,w)$ spanned by $\{E_{v(i)v(j)}:i< j,vs_{ij}\not\preceq  w\}$. Thus, $T_{e_v}^\vee X(w)$ has basis
\begin{align*}
    \{e_{v(i)v(j)}:i< j,vs_{ij}\preceq w\},
\end{align*}
where $E_{v(i)v(j)}$ is the class of $E_{v(i)v(j)}$ modulo $K(v,w)$.

We define the \textit{conormal space} $N^\vee _{X(w),gB}C(v)$ to $C(v)\subseteq X(w)$ at $gB\in C(v)$ as the subspace of $T_{gB}^\vee X(w)$ which annihilates $T_{gB}C(v)\subseteq T_{gB}X(w)$. For $gB=e_v$, $T_{gb}C(v)$ is a $B_v$-submodule of $T_{gB}X(w)$. We write $N^\vee _{X(w)}C(v)$ for the conormal bundle to $C(v)$ in $X(w)$, whose fiber at $gB$ is $N^\vee _{X(w),gB}C(v)$. By definition, there is a short exact sequence of $B_v$-modules
\begin{align*}
    0\to N^\vee_{X(w),e_v}C(v)\to T_{e_v}^\vee X(w)\to T_{e_v}^\vee C(v)\to 0.
\end{align*}
We write
\begin{align*}
    R(v,w)=\{(i,j)\in [n]^2:i< j, \quad v\prec vs_{ij}\preceq w\}.
\end{align*}
Then $N^\vee_{X(w),e_v}C(v)$ is the $B_v$-submodule of $T_{e_v}^\vee X(w)$ with basis $\{e_{v(i)v(j)}:(i,j)\in R(v,w)\}$. 

\begin{definition}
    For $x\in N^\vee_{X(w),e_v}C(v)$, we write $x_{v(i)v(j)}\in \mathbb C$ such that
    \begin{align*}
        x=\sum_{(i,j)\in R(v,w)}x_{v(i)v(j)}e_{v(i)v(j)}.
    \end{align*}
    We use the convention that $x_{v(i)v(j)}=0$ if $(i,j)$ is not in $R(v,w)$. We write $\vec 1$ for the vector $x\in N^\vee_{X(w),e_v}C(v)$ such that every $x_{v(i)v(j)}$ is $1$.
\end{definition}

We also write $S(v)=\{(i,j)\in[n]^2:i\leq j,v(i)\leq v(j)\}$ for the pairs $(i,j)$ which index the basis of $\mathfrak b_v\subseteq \mathfrak b$. Explicitly, $\mathfrak b_v$ is spanned by $\{E_{v(i)v(j)}:(i,j)\in S(v)\}$. When $i< j$, note that $(i,j)\in S(v)$ if and only if $v\preceq vs_{ij}$ \cite[Theorem 2.1.5]{BjornerBrenti}.

With $S(v)$ and $R(v,w)$ established, we introduce a matrix whose rank detects open orbits of $B_v$ in the $N^\vee_{X(w),e_v}C(v)$. 

\begin{definition}
    Let $x\in N^\vee _{X(w),e_v}C(v)$. We define $M^{v,w}(x)$ as a matrix with rows indexed by $R(v,w)$ and columns indexed by $S(v)$. Its entries are given by
    \begin{align*}
        M^{v,w}(x)_{(a,b),(r,s)}&=\begin{cases}
        x_{v(s)v(b)} & \text{if } a=r \text{ and } (s,b)\in R(v,w),\\
        -x_{v(a)v(r)} & \text{if } s=b \text{ and } (a,r)\in R(v,w),\\
        0 & \text{otherwise.}
        \end{cases}
    \end{align*}
\end{definition}

The following gives a purely linear algebraic characterization of the existence of an open $B$-orbit in $N^\vee_{X(w)}C(v)$, and was established in different language in \cite{Lapid_2018} (cf. Remark 4.11 and the surrounding material).
\begin{theorem}\label{thm:matrixcharacterization}
    The following are equivalent:
    \begin{enumerate}
        \item The action of $B$ on the conormal bundle $N^\vee _{X(w)}C(v)$ admits an open orbit. 
        \item There exists some $x\in N^\vee _{X(w),e_v}C(v)$ such that $\rank(M^{v,w}(x))=|R(v,w)|$.
    \end{enumerate}
\end{theorem}

\begin{proof} 
    Because $B$ acts transitively on $C(v)$, condition (1) is equivalent to the condition that $B_v$ has an open orbit on $N^\vee _{X(w),e_v}C(v)$. This is equivalent to the condition that
    \begin{align*}
        [\mathfrak b_v,x]=N^\vee _{X(w),e_v}C(v).
    \end{align*}
    
    By direct computation, we have
    \begin{align*}
        [E_{v(r)v(s)},x]&=\sum_{(a,b)\in R(v,w)}x_{v(a)v(b)}[E_{v(r)v(s)},e_{v(a)v(b)}]\\
        &=\sum_{(a,b)\in R(v,w)}x_{v(a)v(b)}(\delta_{as}e_{v(r)v(b)}-\delta_{br} e_{v(a)v(s)}).
    \end{align*}
    Therefore,
    \begin{align*}
        [E_{v(r)v(s)},x]_{v(a)v(b)}&=\delta_{ar}x_{v(s)v(b)}-\delta_{sb}x_{v(a)v(r)}\\
        &=M^{v,w}(x)_{(a,b),(r,s)}.
    \end{align*}
    In particular, our matrix $M^{v,w}(x)$ has full rank $|R(v,w)|$ if and only if $[\mathfrak b_v,x]$ is the full conormal space $N^\vee_{X(w),e_v}C(v)$. 
\end{proof}

\subsection{Bruhat Pictures}

Before continuing, we introduce Bruhat pictures and related definitions from \cite{billey2001maximalsingularlocischubert}, which allow us to visualize the set $R(v,w)$ and give a pictorial characterization of maximal singular points. We will then recall several results on Bruhat pictures and from \cite{billey2001maximalsingularlocischubert}.

\begin{definition}
    Given $v\preceq w$ in $S_n$, the Bruhat picture for the pair $(v,w)$ is the $n$ by $n$ grid with nodes marked by filled and unfilled dots. For $i\in [n]$, writing $\pt_v(i)$ for the point $(i,v(i))\in \mathbb Z^2$. We place a filled dot $\bullet$ at $\pt_v(i)$, an unfilled dot $\circ$ at $\pt_w(i)$, and a filled dot overlaid on an unfilled one $\bigcircledbullet$ whenever $w(i)=v(i)$. We denote by $\Br(v,w)$ the Bruhat picture of the pair $(v,w)$.
\end{definition}

\begin{remark}\label{rem:bruhatflipping}
    It is clear from the definitions of Bruhat pictures and $w_0$ that we can obtain $\Br(w_0w,w_0v)$ by reflecting $\Br(v,w)$ over a vertical axis, replacing the filled dots with unfilled dots, and replacing the unfilled dots with filled dots.
\end{remark}

For instance, below in \Cref{fig:examplepairwithoutshading} is the plot for $v=[1,2,5,3,4]$ and $w=[2,4,5,3,1]$. 

In general, we only work with particular subsets of the Bruhat picture, labeling only the rows $i$ and points $\pt_v(i),\pt_w(i)$ of interest.

\begin{definition}
    We define the rank function $r_w(p,q):=\#\{i\leq p:w(i)\geq q\}$ for $p,q\in \mathbb Z$, and the difference function $d_{v,w}(p,q):=r_w(p,q)-r_v(p,q)$.
\end{definition}

In $\Br(v,w)$, we shade the points $(p,q)$ in $\mathbb Z^2$ such that $d_{v,w}(p,q)$ is positive. For instance, below in \Cref{fig:examplepairwithshading} is the plot of the pair $v$ and $w$ in \Cref{fig:examplepairwithoutshading}, now with shading.

\begin{figure}[ht]
    \centering

    \begin{subfigure}[t]{0.31\textwidth}
        \centering
        \includegraphics[width=\linewidth]{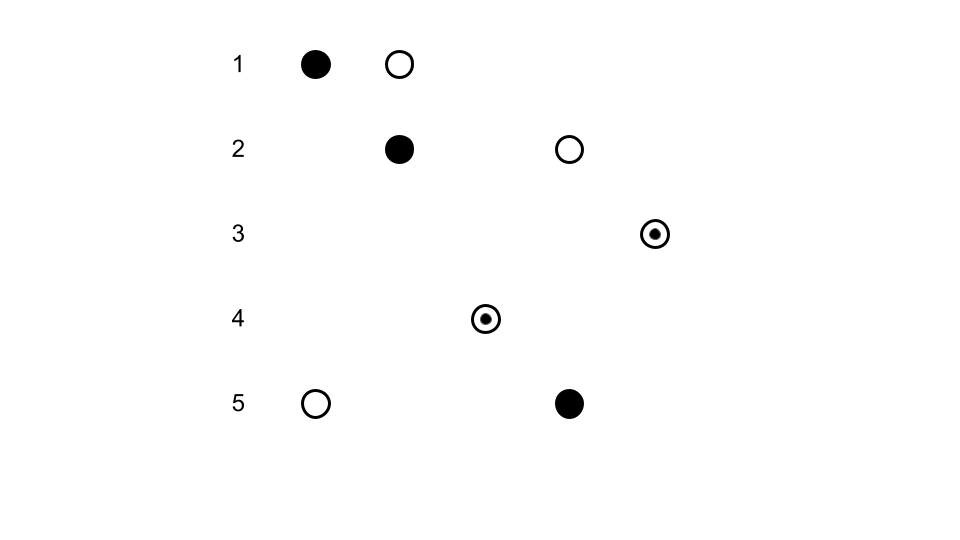}
        \caption{The Bruhat picture $\Br(v,w)$ for
        $v=[1,2,5,3,4]$ and $w=[2,4,5,3,1]$.}
        \label{fig:examplepairwithoutshading}
    \end{subfigure}
    \hfill
    \begin{subfigure}[t]{0.31\textwidth}
        \centering
        \includegraphics[width=\linewidth]{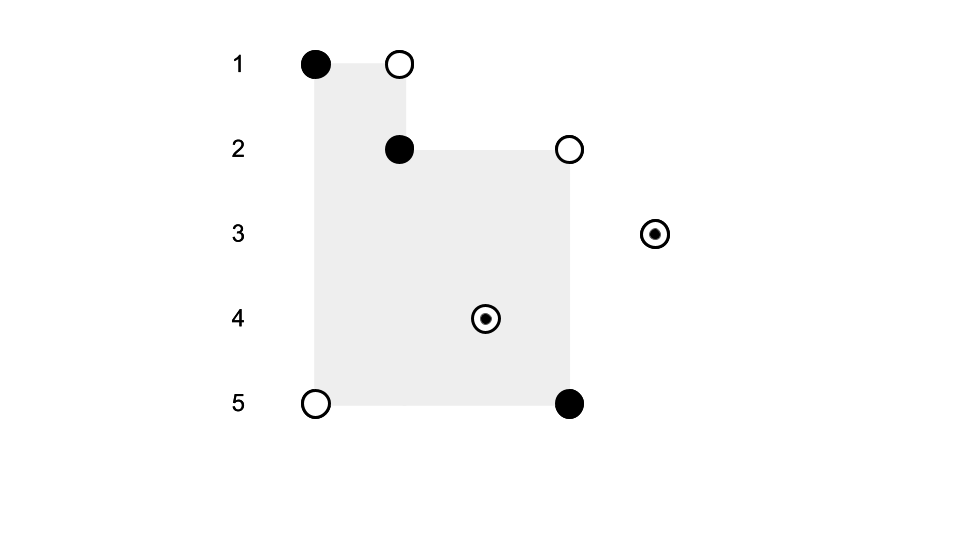}
        \caption{The same Bruhat picture with shading. On the shaded region,
        we have $d_{v,w}\equiv 1$, and outside it, we have $d_{v,w}\equiv 0$.}
        \label{fig:examplepairwithshading}
    \end{subfigure}
    \hfill
    \begin{subfigure}[t]{0.31\textwidth}
        \centering
        \includegraphics[width=\linewidth]{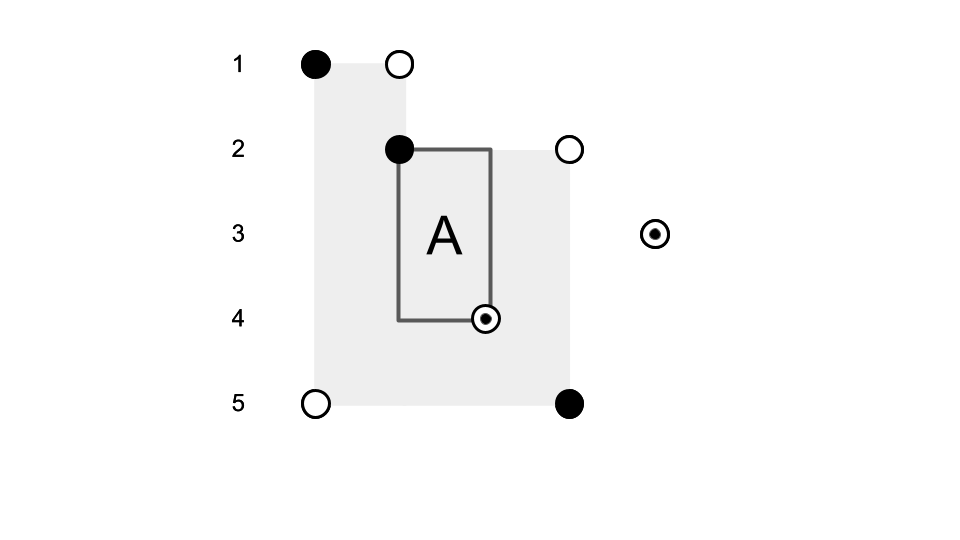}
        \caption{The same Bruhat picture with the region
        $\mcA_{24}(v)$ indicated as region $A$.}
        \label{fig:examplerectangle}
    \end{subfigure}

    \caption{Bruhat pictures for
    $v=[1,2,5,3,4]$ and $w=[2,4,5,3,1]$.}
    \label{fig:examplepairshading}
\end{figure}

\begin{definition}
    For $i<j$, we write $\mcA_{ij}(v)=(i,j)\times (v(i),v(j))\subseteq \mathbb Z^2$. 
\end{definition}

We now recall some general results on properties of $d_{v,w}$ from \cite{billey2001maximalsingularlocischubert}.

\begin{lemma}[{\cite[Lemma 3, Fact 7]{billey2001maximalsingularlocischubert}}]\label{lem:visuallydetermineR}
The following properties of $d_{v,w}$ hold.
    \begin{enumerate}
        \item We have $v\preceq w$ if and only if $d_{v,w}(p,q)\geq 0$ for all $p,q\in \mathbb Z$.
        \item Let $v\preceq w$ and $i\leq j$ such that $v\prec vs_{ij}$ (i.e., $v(i)< v(j)$). Then $(i,j)\in R(v,w)$ if and only if $d_{v,w}(p,q)$ is positive for $(p,q)\in \mcA_{ij}(v)$.
    \end{enumerate}
\end{lemma}

This lemma allows us read off the set $R(v,w)$ from $\Br(v,w)$. In particular, we have $(i,j) \in R(v,w)$ if and only if the box $\mcA_{ij}(v)$ (whose top left corner is $\pt_v(i)$ and bottom right corner is $\pt_v(j)$) is shaded. Applying this to \Cref{fig:examplepairwithshading}, we have by inspection of the Bruhat picture that $v\preceq w$ and that $R(v,w)=\{(1,2),(2,4),(2,5),(4,5)\}$.

\begin{definition}
    Define the set of \textit{maximal singular points} $\maxsing(w)$ to be the set of $v\preceq w$ such that $X(w)$ is singular at $e_v$, and for any $v'$ with $v\prec v'\preceq w$, $X(w)$ is smooth at $e_{v'}$. 
\end{definition}

The main theorem of \cite{billey2001maximalsingularlocischubert} is a characterization of $\maxsing(w)$.

\begin{theorem}[Billey-Warrington]\label{thm:maxsingclassification}
    We have $v\in \maxsing(w)$ if and only if $v=wc$, with $c$ a cycle of the form $(\alpha_1,\dots,\alpha_m,\beta_k,\dots,\beta_1)$, satisfying the following conditions. We emphasize that the permutation $c$ here is written in cycle notation, contrary to our prevailing convention. First, we put the following requirements on $c$.
    \begin{align*}
        1&\leq \alpha_1< \dots< \alpha_m\leq n,\text{ with }w(\alpha_1)> \cdots > w(\alpha_m), \text{ and }\\
        1&\leq \beta_1< \dots< \beta_k\leq n,\text{ with }w(\beta_1)> \cdots > w(\beta_k),
    \end{align*}
    Second, one of the following cases must hold.
    \begin{enumerate}
        \item The 4231 Case: $k,m\geq 2$, $w(\alpha_m) > w(\beta_1)$, and
        \begin{align*}
            \alpha_1&< \beta_1,\dots,\beta_{k-1}< \alpha_2,\dots,\alpha_m< \beta_k.
        \end{align*}
        Also, in \Cref{fig:4231withshading}, which depicts a subset of $\Br(v,w)$ in this case, the shaded region contains no filled or unfilled dots.
        \item The 3412 Case: $k,m\geq 2$
        \begin{align*}
            w(\alpha_{m-1})> w(\beta_1)> w(\alpha_m)> w(\beta_2),\text{ and}\\
            \beta_1,\dots,\beta_{k-1}< \alpha_1< \beta_k< \alpha_2,\dots,\alpha_m.
        \end{align*}
        Also, the shaded region of \Cref{fig:3412withshading} contains no filled or unfilled dots.
        \item The 45312 Case: $k=m=2$, 
        \begin{align*}
            w(\alpha_1)> w(\beta_1)> w&(\alpha_2)> w(\beta_2),\text{ and}\\
            \beta_1< \alpha_1&< \beta_2< \alpha_2.
        \end{align*}
        Also, if we have filled dots in the shaded region of \Cref{fig:45312withshading}, then they are concentric circles, and they are in decreasing order as indicated.
    \end{enumerate}
\end{theorem}

While the above cases are somewhat difficult to parse, the pairs $v,w$ can be described in terms of their Bruhat pictures. In the 4231 Case, we have \Cref{fig:4231withshading}. In the 3412 Case, we have \Cref{fig:3412withshading}. Finally, in the 45312 Case, we have \Cref{fig:45312withshading}. Importantly, the shaded regions below have no points in them (except for the concentric circles in the 45312 Case). The dotted lines in these figures simply indicate that the pattern repeats.

\begin{figure}[ht]
    \centering
    \begin{subfigure}[t]{0.47\textwidth}
        \centering
        \includegraphics[width=\linewidth]{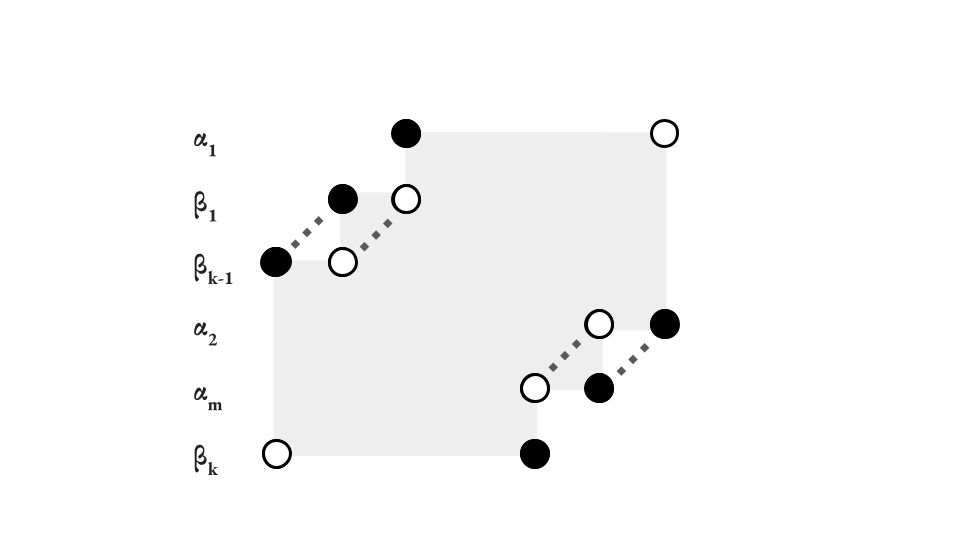}
        \caption{The Bruhat picture $\Br(v,w)$ with $v\in \maxsing(w)$ in the 4231 Case. On the shaded region, $d_{v,w}\equiv 1$, and outside of the shaded region, $d_{v,w}\equiv 0$. The shaded region contains no filled or unfilled points. By the dotted lines, we mean a continuation of the staircase pattern.}
        \label{fig:4231withshading}
    \end{subfigure}
    \hfill
    \begin{subfigure}[t]{0.47\textwidth}
        \centering
        \includegraphics[width=\linewidth]{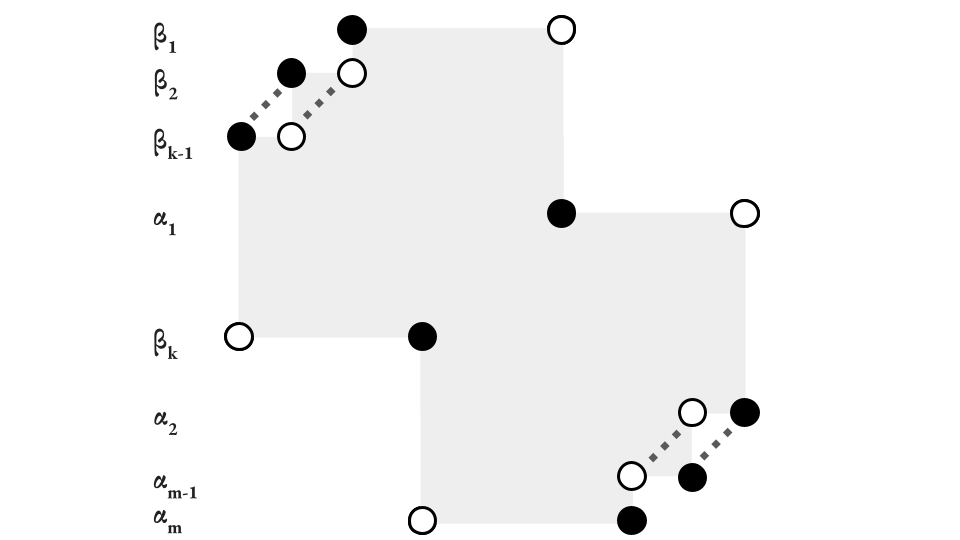}
        \caption{The Bruhat picture $\Br(v,w)$ with $v\in \maxsing(w)$ in the 3412 Case. On the shaded region, $d_{v,w}\equiv 1$, and outside of the shaded region, $d_{v,w}\equiv 0$. The shaded region contains no filled or unfilled points. By the dotted lines, we mean a continuation of the staircase pattern.}
        \label{fig:3412withshading}
    \end{subfigure}
    \caption{}
    \label{fig:billeywarringtoncases}
\end{figure}

\begin{figure}[h]
\begin{center}
    \includegraphics[width=8cm]{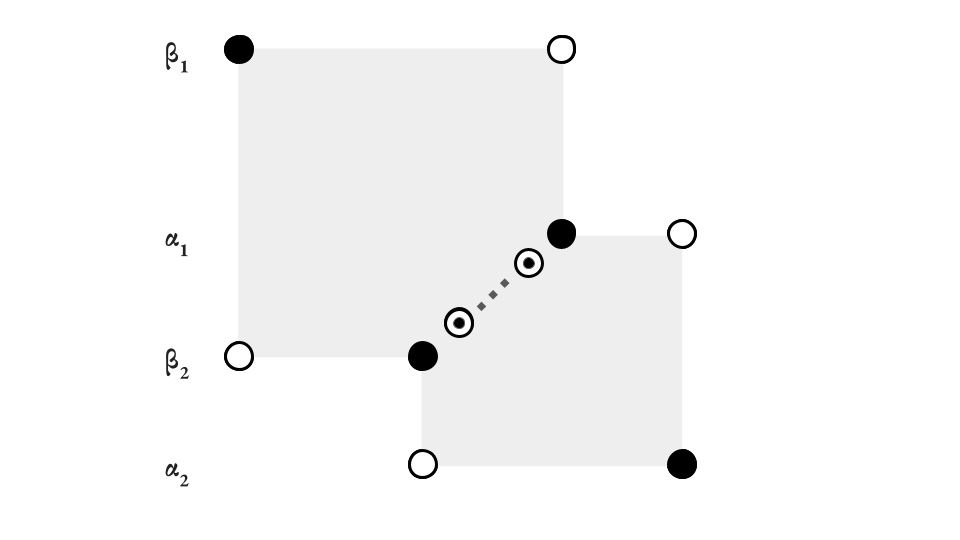}
\end{center}
\caption{The Bruhat picture $\Br(v,w)$ with $v\in \maxsing(w)$ in the 45312 Case. On the shaded region, $d_{v,w}\equiv 1$, and outside of the shaded region, $d_{v,w}\equiv 0$. The shaded region contains no filled or unfilled points other than the indicated concentric circles. By the dotted line, we mean that the concentric circles continue down and to the left in order. That is, if $\pt_i(v)$ and $\pt_j(v)$ are two of these dots with $i<j$, then $\pt_i(v)$ is to the top right of $\pt_j(v)$.}
\label{fig:45312withshading}
\end{figure}

\section{Maximal Singular Points}\label{sec:maxsing}

In this section, we prove that if $w_0w\in \maxsing(w_0v)$, then there is no open $B_v$-orbit in $N^\vee_{X(w),e_v}C(v)$. We will do this via a directed graph $\mcG_1(v,w)$. We address the 4231 and 3412 Cases in \Cref{sec:maximalsingular4}, and we address the 45312 Case in \Cref{sec:maxsing5}. In \Cref{sec:maxsing5}, we use this to conclude that if $X(w_0v)$ is singular at $e_{w_0w}$, then there is no open $B_v$-orbit in $N^\vee_{X(w),e_v)}C(v)$. 

\subsection{Maximal Singular Points in the 4231 and 3412 Cases}\label{sec:maximalsingular4}

In this subsection, we will prove that for any $v\preceq w$ such that $w_0w\in \maxsing(w_0v)$ in the 4231 or 3412 Cases, the rank of $M^{v,w}(x)$ is at most $|R(v,w)|-1$, and thus that there is no open orbit in $N^\vee_{X(w),e_v}C(v)$.

We will now introduce a graph $\mcG_1(v,w)$ such that a cycle in the graph gives us information about the rank of $M^{v,w}(x)$.

\begin{definition}
    Define the partial order $\trianglelefteq$ on $R(v,w)$ such that $(i,j)\trianglelefteq(i',j')$ if and only if we have the containment of rectangles $\mcA_{ij}(v)\subseteq \mcA_{i'j'}(v)$.
\end{definition}

We let
\begin{align*}
    E_0(v,w)=\{(i,i):(i,j)\in R(v,w)\text{ or }(j,i)\in R(v,w)\text{ for some }j\},
\end{align*}
and let $E_1(v,w)$ be the elements of $R(v,w)$ which are minimal with respect to $\trianglelefteq$. Visually, $(i,j)$ is in $E_1(v,w)$ if there are no filled dots in the rectangle $\mcA_{ij}(v)$. We define the directed graph $\mcG_1(v,w)$, with vertices labeled by $E_0(v,w)$ and edges labeled by $E_1(v,w)$, such that the source of the edge $(i,j)$ is $(i,i)$ and the target is $(j,j)$. This turns out to be the same, up to a choice of convention, as the atom graph considered by Tsukerman--Williams in \cite{TsukermanWilliamsBruhatIntervalPolytopes}. By construction, the graph $\mcG_1(v,w)$ has no self-edges or multi-edges.

For example, in our previous example \Cref{fig:examplepairwithshading}, we have $(2,4)\trianglelefteq (2,5)$ and $(4,5)\trianglelefteq (2,5)$. Thus, we have $E_0(v,w)=\{(1,1),(2,2),(4,4),(5,5)\}$ and $E_1(v,w)=\{(1,2),(2,4),(4,5)\}$, so $\mcG_1(v,w)$ is a directed path of length 3. 

\begin{definition}
    For a directed graph $\mcG=(V,E)$, the signed incidence matrix is the matrix $A_\mathcal G$ with rows indexed by $V$ and columns indexed by $E$ such that
    \begin{align*}
        (A_\mathcal G)_{v_1,(v_2,v_3)}=\begin{cases}
            1&\text{if }v_1=v_2\\
            -1&\text{if }v_1=v_3\\
            0&\text{otherwise}.
        \end{cases}
    \end{align*}

\end{definition}

We recall the standard formula for the rank of the signed incidence matrix {\cite[Theorem 8.3.1]{godsil01}}.

\begin{theorem}\label{thm:incidencematrixrank}
    Let $\mcG=(V,E)$ be a simple undirected graph with $k$ connected components. Then the rank of $A_\mathcal G$ is $|V|-k$.
\end{theorem}

For a directed graph $\mcG=(V,E)$, we denote by $\overline \mcG$ the undirected graph whose edge relation is the symmetric closure of $E$. 

\begin{corollary}\label{cor:allcomponentstree}
    The rank of the signed incidence matrix of $\mcG=(V,E)$ is $|E|$ (i.e., the signed incidence matrix has full rank) if and only if every connected component of $\overline \mcG$ is a tree.
\end{corollary}

\begin{proof}
    For each connected component $C$, we have $|E(C)|=|V(C)|-1$ if and only if $C$ is a tree. Summing over the connected components yields the corollary.
\end{proof}

\begin{corollary}\label{cor:nocycle}
    If $\overline{\mcG_1(v,w)}$ contains a cycle, then $M^{v,w}(x)$ does not have full rank for any $x\in N^\vee_{X(w),e_v}C(v)$.
\end{corollary}

\begin{proof}
    For $(a,b) \in E_1(v,w)$, observe that 
    \begin{align*}
        M^{v,w}(x)_{(a,b),(r,s)}
        =\begin{cases} 
                x_{v(a)v(b)} & \text{if $(r,s) = (a,a)$} \\
                -x_{v(a)v(b)} & \text{if $(r,s) = (b,b)$} 
                 \\
                 0 & \text{otherwise.}
        \end{cases} 
    \end{align*}
    Thus, the submatrix $D$ of $M^{v,w}(x)$ with rows indexed by $E_1(v,w)$ and columns indexed by $E_0(v,w)$ is (up to scaling of rows) the transpose of the signed incidence matrix of $\mcG_1(v,w)$. In particular, the rank of $D$ is at most the rank of the signed incidence matrix. In fact, because every nonzero entry of $M^{v,w}(x)$ in the rows $ E_1(v,w)$ belongs to the submatrix $D$, then $M^{v,w}(x)$ does not have full rank if $D$ does not have full rank. Now, if there is a cycle in $\overline{\mcG_1(v,w)}$, the signed incidence matrix of $\mcG_1(v,w)$ does not have full rank by \Cref{cor:allcomponentstree}. Thus $D$ does not have full rank, so the result follows.
\end{proof}

We will use the preceding corollary to show that if $v\preceq w$ such that $w_0w\in \maxsing(w_0v)$ in the 4231 or 3412 Cases, then $M^{v,w}(x)$ does not have full rank for any $x\in N^\vee_{X(w),e_v}C(v)$. We will do this by showing that in either case, the graph $\overline{\mcG_1(v,w)}$ contains a cycle.

\begin{lemma}\label{lem:4231maxsing}
    Let $v\preceq w$ such that $w_0w\in \maxsing(w_0v)$ in the 4231 Case. Then $M^{v,w}(x)$ does not have full rank for any $x$ in $N^\vee_{X(w),e_v}C(v)$.
\end{lemma}

\begin{proof}
    Observe that $\Br(w_0w,w_0v)$ has the form of \Cref{fig:4231withshading}. Thus, the Bruhat picture $\Br(v,w)$ is as in \Cref{fig:flipped4231withshading}, via \Cref{rem:bruhatflipping}.  
    
    % \begin{figure}[h]
    % \begin{center}
    %     \includegraphics[width=8cm]{Images/Flipped4231WithoutShading.png}
    % \end{center}
    % \caption{}
    % \label{fig:flipped4231withoutshading}
    % \end{figure}

    \begin{figure}[h]
    \begin{center}
        \includegraphics[width=8cm]{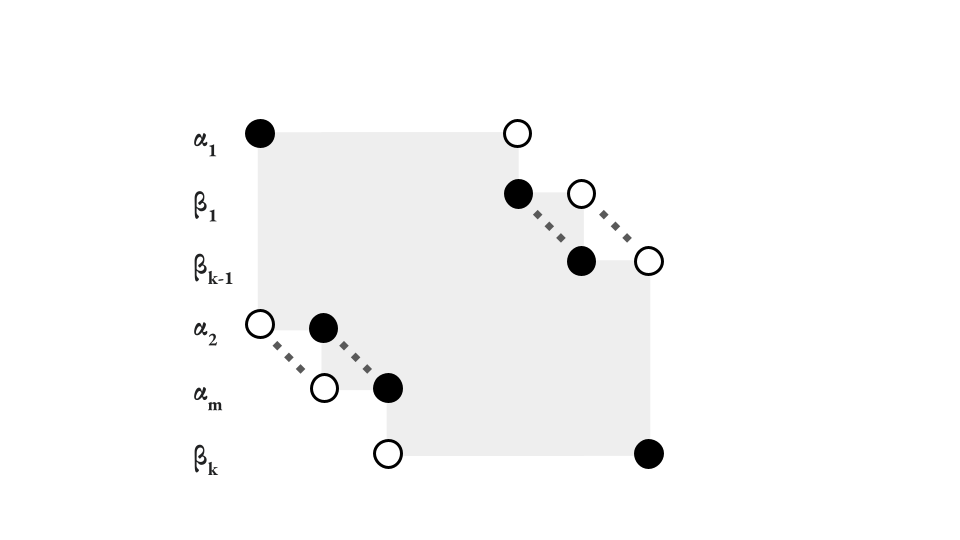}
    \end{center}
    \caption{The Bruhat picture $\Br(v,w)$ with $w_0w\in \maxsing(w_0v)$ in the 4231 Case. On the shaded region, $d_{v,w}\equiv 1$, and outside of the shaded region, $d_{v,w}\equiv 0$. The shaded region contains no filled or unfilled points.}
    \label{fig:flipped4231withshading}
    \end{figure}

    Recalling that the shaded region in \Cref{fig:4231withshading} has no points in it, the shaded region of \Cref{fig:flipped4231withshading} has no points in it. Applying \Cref{lem:visuallydetermineR}, we can see
    \begin{align*}
        R(v,w)&=\{(\beta_i,\beta_{i+1}):1\leq i\leq k-1\}\\
        &\cup \{(\alpha_i,\alpha_{i+1}):1\leq i\leq m-1\}\cup \{(\alpha_1,\beta_1),(\alpha_m,\beta_k)\}.
    \end{align*}
    
    Because the shaded area contains no dots, each of the above is minimal with respect to $\trianglelefteq$, and thus in $E_1(v,w)$. Therefore, the graph $\overline{\mcG_1(v,w)}$ is itself a cycle, and so $M^{v,w}(x)$ does not have full rank by \Cref{cor:nocycle}.
\end{proof}

\begin{lemma}\label{lem:3412maxsing}
    Let $v\preceq w$ such that $w_0w\in \maxsing(w_0v)$ in the 3412 Case. Then $M^{v,w}(x)$ does not have full rank for any $x\in N^\vee_{X(w),e_v}C(v)$.
\end{lemma}

\begin{proof}
    If $w_0w\in \maxsing(w_0v)$ in the 3412 case, the Bruhat picture $\Br(v,w)$ is as in \Cref{fig:flipped3412withshading}, by applying \Cref{rem:bruhatflipping} to \Cref{fig:3412withshading}. This lets us determine
    \begin{align*}
        R(v,w)&=\{(\beta_i,\beta_{i+1}):1\leq i\leq k-1\}\\
        &\cup \{(\alpha_i,\alpha_{i+1}):1\leq i\leq m-1\}\cup \{(\alpha_1,\beta_k),(\beta_1,\alpha_m)\},
    \end{align*}
    and so by the same argument we made in \Cref{lem:4231maxsing}, there is a cycle in $\overline{\mcG_1(v,w)}$.
\end{proof}

    \begin{figure}[h]
    \begin{center}
        \includegraphics[width=8cm]{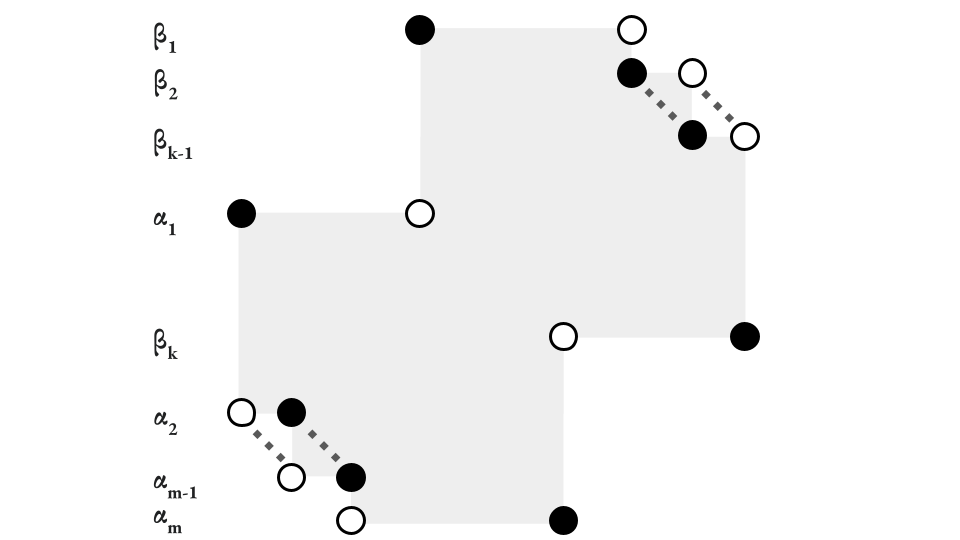}
    \end{center}
    \caption{The Bruhat picture $\Br(v,w)$ with $w_0w\in \maxsing(w_0v)$ in the 3412 Case. On the shaded region, $d_{v,w}\equiv 1$, and outside of the shaded region, $d_{v,w}\equiv 0$. The shaded region contains no filled or unfilled points.}
    \label{fig:flipped3412withshading}
    \end{figure}

\subsection{Maximal Singular Points in the 45312 Case}\label{sec:maxsing5}

In this subsection, we address the maximal singular points in the 45312 Case. 

We first recall some results from \cite{billey2001maximalsingularlocischubert}.

\begin{definition}
    For $w\in S_n$ and $Z=\{z_1,\dots,z_k\}\subseteq \{1,\dots,n\}$ such that $z_1< \dots< z_k$, we denote by $\fl_Z(w)\in S_{|Z|}$ the flattened permutation such that $\fl_Z(w)(i)< \fl_Z(w)(j)$ if and only if $w(z_i)<w(z_j)$. Moreover, for each pair $v\preceq w$, we define $\overline{v}^w=\fl_{E_0(v,w)}(v)$ and $\overline{w}^v=\fl_{E_0(v,w)}(w)$.
\end{definition}

For instance, if $w=[2,4,5,3,1]$ and $Z=\{2,3,4\}$, then $\fl_Z(w)=[2,3,1]$.

The following lemma says that all relevant properties of the pair $(v, w)$ are shared by the flattened pair $(\overline v^w,\overline w^v)$.

\begin{lemma}[{\cite[Lemma 15, Proposition 16]{billey2001maximalsingularlocischubert}}]\label{lem:preservesmatrix}
    We have $\overline v\preceq \overline w$ if and only if $v\preceq w$. Also, there exists a bijection between $R(v,w)$ and $R(\overline v^w,\overline w^v)$ which preserves $\trianglelefteq$. In particular, up to indexing of rows and columns, $M^{v,w}(x)$ is the same as $M^{\overline v^w,\overline w^v}(x)$. 
\end{lemma}

\begin{lemma}\label{lem:45312maximal}
    Let $v\preceq w$ such that $w_0w\in \maxsing(w_0v)$ in the 45312 Case. Then $M^{v,w}(x)$ does not have full rank for any $x\in N^\vee_{X(w),e_v}C(v)$.
\end{lemma}

\begin{proof}
    Using the Bruhat picture from \Cref{fig:45312withshading} and \Cref{rem:bruhatflipping}, we see that $\Br(v,w)$ is given by \Cref{fig:flipped45312withshading}.

    % \begin{figure}[h]
    % \begin{center}
    %     \includegraphics[width=8cm]{Images/Flipped45312WithoutShading.png}
    % \end{center}
    % \caption{}
    % \label{fig:flipped45312withoutshading}
    % \end{figure}

    \begin{figure}[h]
    \begin{center}
        \includegraphics[width=8cm]{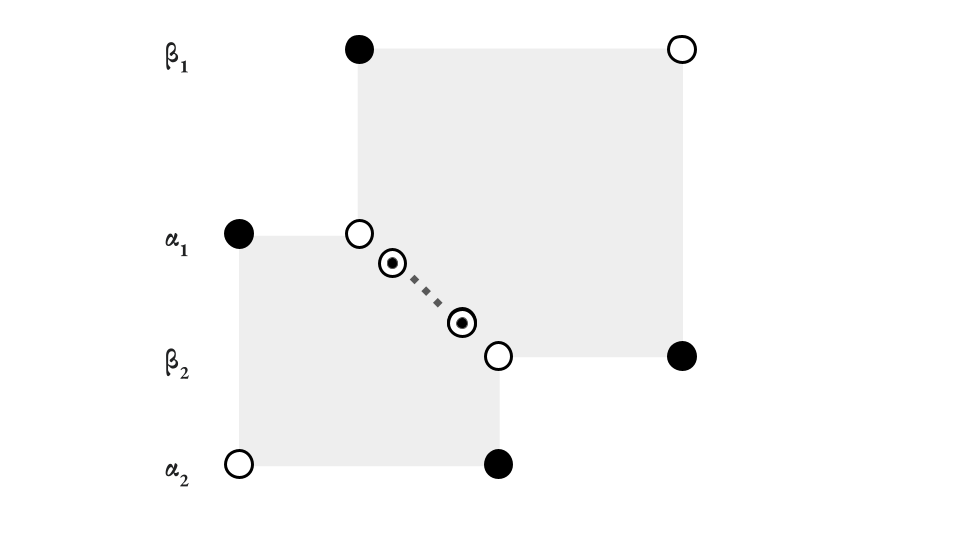}
    \end{center}
    \caption{The Bruhat picture $\Br(v,w)$ with $w_0w\in \maxsing(w_0v)$ in the 45312 Case. On the shaded region, $d_{v,w}\equiv 1$, and outside of the shaded region, $d_{v,w}\equiv 0$. The shaded region contains no filled or unfilled points other than the indicated concentric circles.}
    \label{fig:flipped45312withshading}
    \end{figure}

    Recalling that the shaded region of \Cref{fig:45312withshading} contains no points besides the concentric circles depicted, the shaded region of \Cref{fig:flipped45312withshading} also contains no points, with the exception of possible concentric circles in $\mcA_{\alpha_1,\beta_2}(w)$. Let $M$ be the set of $\mu\in (\alpha_1,\beta_2)$ such that $\pt_v(\mu)=\pt_w(\mu)$ is one of said concentric circles, i.e.
    \begin{align*}
        M&=\{\mu \in (\alpha_1,\beta_2):v(\alpha_1) < v(\mu) = w(\mu) < v(\beta_2)\}.
    \end{align*}
    Enumerate $M=\{\mu_1,\dots,\mu_n\}$ such that $\mu_1<\dots<\mu_n$. We have
    \begin{align*}
        R(v,w)&=\{(\beta_1,\mu_i):\mu_i\in M\}\cup \{(\alpha_1,\mu_i):\mu_i\in M\}\\
        &\quad \cup \{(\mu_i,\beta_2):\mu_i\in M\}\cup \{(\mu_i,\alpha_2):\mu_i\in M\}\\
        &\quad \cup\{(\mu_i,\mu_j):i< j\}\cup \{(\beta_1,\beta_2),(\beta_1,\alpha_2),(\alpha_1,\beta_2),(\alpha_1,\alpha_2)\},
    \end{align*}
    so
    \begin{align*}
        E_0(v,w)=\{\beta_1,\alpha_1,\beta_2,\alpha_2\}\cup M.
    \end{align*}
    Thus $\Br(\overline v,\overline w)$ is depicted in \Cref{fig:overlinevoverlinew} by \Cref{lem:preservesmatrix}. Furthermore, the Bruhat picture $\Br(\overline v^w,w_0)$ of flattened permutations in $S_{|M|+4}$ is as in \Cref{fig:overlinevw_0} (strictly speaking, what is depicted is the case where $|M|$ is odd, but the difference is immaterial). 

    \begin{figure}[ht]
        \centering
        \begin{subfigure}[t]{0.47\textwidth}
            \centering
            \includegraphics[width=\linewidth]{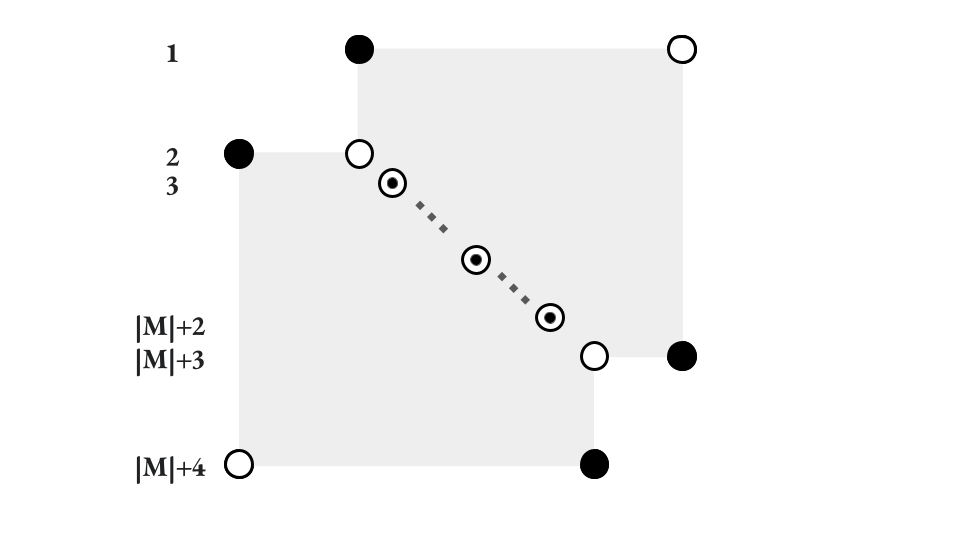}
            \caption{The Bruhat picture $\Br(\overline v^w,\overline w^v)$ with $w_0w\in \maxsing(w_0v)$ in the 45312 Case. On the shaded region, $d_{\overline v^w,\overline w^v}\equiv 1$, and outside of the shaded region, $d_{\overline v^w,\overline w^v}\equiv 0$.}
            \label{fig:overlinevoverlinew}
        \end{subfigure}
        \hfill
        \begin{subfigure}[t]{0.47\textwidth}
            \centering
            \includegraphics[width=\linewidth]{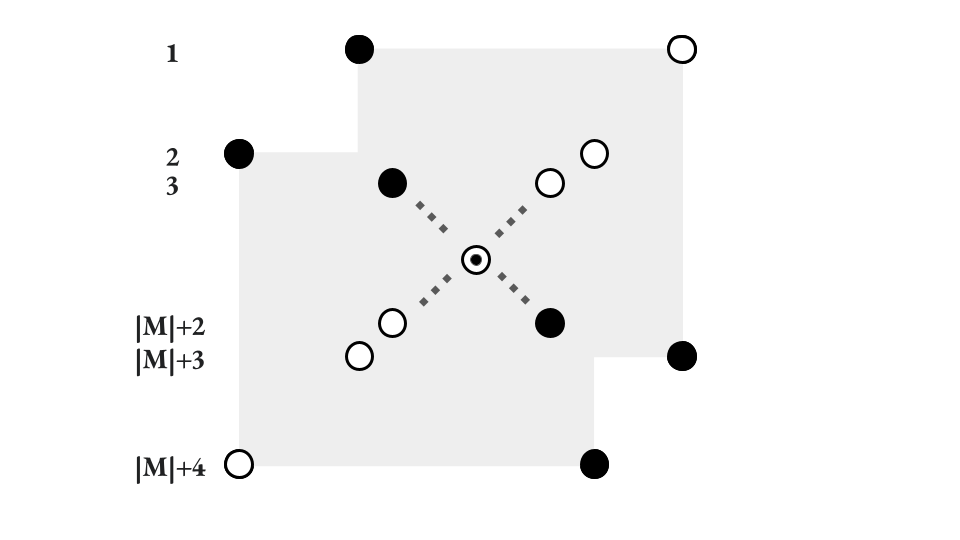}
            \caption{The Bruhat picture $\Br(\overline v^w,w_0)$ with $w_0w\in \maxsing(w_0v)$ in the 45312 Case. On the shaded region, $d_{\overline v^w,w_0}\geq 1$, and outside of the shaded region, $d_{\overline v^w,w_0}\equiv 0$. By the dotted lines going from the top left to bottom right (resp. top right to bottom left), we mean that the filled dots continue down and to the right (resp. down and to the left) in order.}
            \label{fig:overlinevw_0}
        \end{subfigure}
        \caption{}
        \label{fig:overlineconfigurations}
    \end{figure}

    We observe that \Cref{fig:overlinevoverlinew} and \Cref{fig:overlinevw_0} have the same shaded regions. Thus, $R(\overline v^w,\overline w^v)=R(\overline v^w,w_0)$, so $M^{\overline v^w,\overline w^v}(x)$ and $M^{\overline v^w,w_0}(x)$ are in fact the same matrix. Recall {\cite[Theorem 1.2]{LapidMinguez2025}}, which says that if $v\preceq w$ and $w$ is $231$-avoiding, then $X(w_0v)$ is smooth at $e_{w_0w}$ if and only if there is an open $B$-orbit in $N^\vee_{X(w)}C(v)$. Strictly speaking, it is stated in the setting of $G$-orbits in $G/B\times G/B$, but the setting of $B$-orbits in $G/B$ is equivalent. The permutation $w_0$ is $231$-avoiding, and $X(w_0\overline v^w)$ is singular at $e_{w_0\overline w^v}$ by \Cref{rem:bruhatflipping} and \Cref{thm:maxsingclassification}. Thus, the matrices $M^{\overline v^w,w_0}(x)=M^{\overline v^w,\overline w^v}(x)$ are the same for all $x\in N^\vee_{X(w),e_v}C(v)$ and are not of full rank. Finally, by \Cref{lem:preservesmatrix}, $M^{v,w}(x)$ is not of full rank for all $x\in N^\vee_{X(w),e_v}C(v)$.
\end{proof}

\begin{corollary}\label{cor:maxsingimpliesnoopen}
    If $v\preceq w$ and $w_0w\in \maxsing(w_0v)$, then the action of $B_v$ on the conormal space $N^\vee _{X(w),e_v}C(v)$ does not admit an open orbit. 
\end{corollary}

\begin{proof}
    By \Cref{thm:maxsingclassification}, we have $w_0w\in \maxsing(w_0v)$ in the 4231 Case, the 3412 Case, or the 45312 Case. These are addressed in turn in \Cref{lem:4231maxsing}, \Cref{lem:3412maxsing}, and \Cref{lem:45312maximal}. 
\end{proof}

\begin{theorem}\label{thm:singularimpliesnoopenorbit}
    Let $v\preceq w$ such that $X(w_0v)$ is singular at $e_{w_0w}$. Then there is no open $B_v$-orbit in $N^\vee _{X(w),e_v}C(v)$.
\end{theorem}

\begin{proof}
    Let $v\preceq w'\preceq w$. Then the natural projection $N^\vee_{X(w),e_v}C(v)\twoheadrightarrow N^\vee_{X(w'),e_v}C(v)$ is surjective, linear, and $B_v$-equivariant. Thus, because both the domain and codomain of this map are finite-dimensional, the map is a projection up to a change of coordinates, and is thus an open map. Therefore, if there is an open $B_v$-orbit in $N^\vee_{X(w),e_v}C(v)$, then there is also an open orbit in $N^\vee_{X(w'),e_v}C(v)$.  Equivalently, if there is no open $B_v$-orbit in $N^\vee _{X(w'),e_v}C(v)$, there is no open $B_v$-orbit in $N^\vee _{X(w),e_v}C(v)$. For every $v\preceq w$ such that $X(w_0v)$ is singular at $e_{w_0w}$, there exists some $w'\preceq w$ such that $w_0w'\in \maxsing(w_0v)$. By \Cref{cor:maxsingimpliesnoopen}, there is no open $B_v$-orbit in $N^\vee _{X(w'),e_v}C(v)$, so there is no open $B_v$-orbit in $N^\vee _{X(w),e_v}C(v)$.
\end{proof}

\section{A Necessary Condition for Non-Existence of an Open $B_v$-orbit}\label{sec:necessarycondition}

\subsection{Defining the Graphs}

In this subsection, we recursively define graphs analogous to $\mcG_1(v,w)$ which capture the remaining information of $R(v,w)$. The idea is to produce a filtration on $R(v,w)$ based on how many elements are less than a fixed element with respect to $\trianglelefteq$. Recall our construction of $E_0(v,w)$ and $E_1(v,w)$ in \Cref{sec:maximalsingular4}. 

\begin{definition}
    Given $E_i(v,w)\subseteq R(v,w)$, we define $E_{i+1}(v,w)$ as the elements of
    \begin{align*}
        R(v,w)\setminus \left(\bigcup_{1\leq j\leq i}E_j(v,w)\right)
    \end{align*}
    which are minimal with respect to $\trianglelefteq$.
\end{definition}

For instance, in \Cref{fig:examplepairwithshading}, we have
\begin{align*}
    E_0(v,w)&=\{(1,1),(2,2),(4,4),(5,5)\},\\
    E_1(v,w)&=\{(1,2),(2,4),(4,5)\},\\
    \intertext{and}
    E_2(v,w)&=\{(2,5)\}.
\end{align*}
% Thus $\mcG_1(v,w)$ is the directed path $(1,2,4,5)$, and $\mcG_2(v,w)$ is the directed edge $(2,5)$ between $(2,4)$ and $(4,5)$.
These are visualized in \Cref{fig:exampleei}, in which the red arrows denote elements of $E_1(v,w)$ and the blue arrow is the element of $E_2(v,w)$.

\begin{figure}[h]
\begin{center}
    \includegraphics[width=8cm]{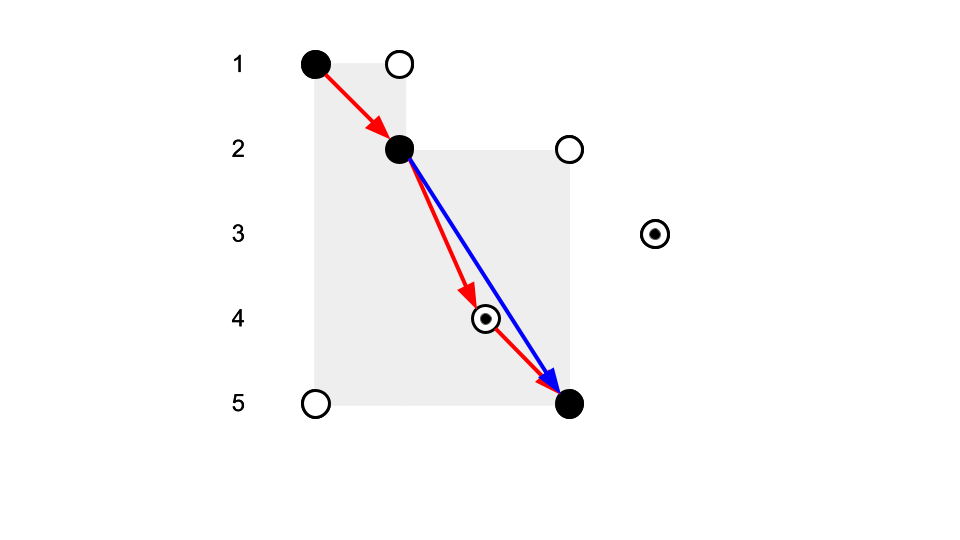}
\end{center}
\caption{The Bruhat picture $\Br(v,w)$ with $v=[1,2,5,3,4]$ and $w=[2,4,5,3,1]$. On the shaded region, $d_{v,w}\equiv 1$, and outside of the shaded region, $d_{v,w}\equiv 0$. For each $(i,j)$ in $E_1(v,w)$, there is a red arrow from $\pt_i(v)$ to $\pt_j(v)$, and for each $(i,j)$ in $E_2(v,w)$, there is a blue arrow from $\pt_i(v)$ to $\pt_j(v)$.}
\label{fig:exampleei}
\end{figure}

Whenever $j> i$, we have $M^{v,w}(x)_{(a,b),(r,s)} = 0$ for any $(a,b) \in E_i(v,w)$ and $(r,s) \in E_j(v,w)$. Thus, after removing all columns which are identically zero and permuting rows and columns, we can write $M^{v,w}(x)$ as a block matrix
\begin{align*}
    M^{v,w}(x)=\begin{bmatrix}
        D_1\\
        *&D_2\\
        *&*&D_3\\
        \vdots&\vdots&\vdots &\ddots\\
        *&*&*&\cdots &D_m
    \end{bmatrix}
\end{align*}
where the rows of $D_i$ are indexed by $E_i(v,w)$ and the columns of $D_i$ are indexed by $E_{i-1}(v,w)$. Here, $m$ is the greatest value such that $E_m(v,w)$ is nonempty.

Thus, we have
\begin{align*}
\rank(M^{v,w}(x))\geq \sum_{1\leq i\leq m}\rank(D_i).
\end{align*}

\begin{lemma}\label{lem:uniquepath}
    If $\overline{\mcG_1(v,w)}$ is a forest, then for every $(a,b)\in E_i(v,w)$, there are exactly $i-1$ values $k$ such that $\pt_v(k)\in \mcA_{ab}(v)$. Further, for every $m_1<m_2$ such that $\pt_v(k_1),\pt_v(k_2)\in \mcA_{ab}(v)$, we have $v(k_1)<v(k_2)$.
\end{lemma}

\begin{proof}
    We induct on $i$. The $i=1$ is clear. Suppose $i\geq 2$, and assume the lemma for every $j<i$. Let $p$ be the minimum $k$ in $(a,b)$ such that $\pt_v(k)\in \mcA_{ab}(v)$. 
    
    If there exists $q$ such that $p<q<b$ and $v(a)<v(q)<v(p)$, then take a minimal such $q$. We have, by minimality, that $(a,p)$ and $(a,q)$ are in $E_1(v,w)$. Further, by the induction hypothesis, $(p,b)$ and $(q,b)$ are in $E_{j_p}(v,w)$ and $E_{j_q}(v,w)$ respectively, with $j_p,j_q<i$. Applying the induction hypothesis to $(p,b)$ and $(q,b)$, we have $j_p-1$ and $j_q-1$ elements $k_1^p,\dots,k_{j_p-1}^p$ and $k_1^q,\dots,k_{j_q-1}^q$ of $\mcA_{pb}(v)$ and $\mcA_{qb}(v)$ respectively. Letting $k_0^p=p$, $k_0^q=q$, and $k_{j_p}^p=k_{j_q}^q=b$, we must have $(k_r^p,k_{r+1}^p)\in E_1(v,w)$ and $(k_r^q,k_r^{q+1})\in E_1(v,w)$ for every $r$, so these give directed paths in $\mcG_1(v,w)$ from $p$ to $b$ and from $q$ to $b$. Together with the edges $(a,p)\in E_1(v,w)$ and $(a,q)\in E_1(v,w)$, there is a cycle in $\overline{\mcG_1(v,w)}$. 
    
    Thus, there is no such $q$, so $(p,b)\in E_{i-1}(v,w)$, because $(a,b)$ is in $E_i(v,w)$ and is a minimal element with respect to $\trianglelefteq$ of
    \begin{align*}
        \{(i,j)\in R(v,w):(p,b)\trianglelefteq (i,j)\}.
    \end{align*}
    The $k_1^p,\dots,k_{j_p-1}^p$ form the remaining $i-2$ points in $\mcA_{ab}(v)$.
\end{proof}

\begin{definition}
    If $\overline{\mcG_1(v,w)}$ is a forest, then for $i\geq 2$, we can define a directed graph $\mcG_i(v,w)$ with vertices labeled by $E_{i-1}(v,w)$ and edges labeled by $E_i(v,w)$, such that the vertex $(r,s)\in E_{i-1}(v,w)$ is the source (resp. target) of the edge $(a,b)$ if and only if $(r,s)\trianglelefteq (a,b)$ and $r=a$ (resp. $s=b$). Observe that $D_i$ is the signed incidence matrix of $\mcG_i(v,w)$ when $x=\vec 1$.
\end{definition}

\begin{proposition}\label{prop:sufficientconditionforopen}
    Let $v\preceq w$. If every connected component of every $\overline{\mcG_i(v,w)}$ is a tree, then $M^{v,w}(\vec 1)$ has full rank.
\end{proposition}

\begin{proof}
    The submatrix $D=D_1$  is the signed incidence matrix of $\mcG_1(v,w)$, and in fact, $D_i$ is the signed incidence matrix of $\mcG_i(v,w)$. Thus, by \Cref{cor:allcomponentstree}, each $D_i$ has full rank, and so $M^{v,w}(\vec 1)$ does as well.
\end{proof}

\subsection{The Necessary Condition}

In this subsection, we show that the graphs $\mcG_j(v,w)$ allow us to characterize the configurations in $\Br(v,w)$ whenever there is no open $B_v$-orbit in $N_{X(w),e_v}^\vee C(v)$. 

\begin{definition}
    Let $\mathcal G=(V,E)$ be a directed graph. The line digraph $L(\mathcal G)$ has vertex set $E$, and its edges are of the form $((v_1,v_2),(v_2,v_3))$ with $(v_1,v_2),(v_2,v_3)\in E$. 
\end{definition}

The definition for undirected graphs is analogous.

We now recall a theorem of Chartrand on line graphs of trees {\cite[Theorem 8.5]{Harary1969GraphTheory}}.

\begin{theorem}[Chartrand]\label{thm:chartrandblock}
    The line graph $L(\mathcal T)$ of an undirected tree $\mathcal T$ is a block graph. That is, every bi-connected subset (a subset $A\subseteq T$ such that for every $v\in A$, the induced subgraph $A\setminus \{v\}\subseteq T$ is connected) is a clique. 
\end{theorem}

The following graph-theoretic lemma will allow us to show that, if $\overline{\mcG_1(v,w)}$ is a forest, then a cycle in $\overline{\mcG_i(v,w)}$ with $i>1$ induces a configuration in $\Br(v,w)$ akin to \Cref{fig:45312withshading}.

\begin{lemma}\label{lem:centralvertex}
    Let $\mathcal G=(V,E)$ be a directed graph such $\overline{\mathcal G}$ is a forest, and let $\mathcal H$ be a directed subgraph of $L(\mathcal G)$ such that $\overline{\mathcal H}$ contains a cycle. Then there is some $c$ in $\mathcal G$ such that every element of the cycle is an edge $(v,c)$ or an edge $(c,v)$ for some $v$ in $\mathcal H$. Further, every edge in the cycle is from a vertex $(v,c)$ to a vertex $(c,v')$.
\end{lemma}

\begin{proof}
    Take some cycle in $\overline{\mathcal H}$ with vertices $(v_1,v_2),\dots,(v_{n-1},v_n)$. Then the set
    \begin{align*}
        \{(v_1,v_2),\dots,(v_{n-1},v_n)\}
    \end{align*}
    is bi-connected, so by \Cref{thm:chartrandblock}, the induced subgraph $\{(v_1,v_2),\dots,(v_{n-1},v_n)\}\subseteq L(\overline{\mathcal G})$ is a clique. Therefore, there exists an element $c$ of $\{v_1,v_2\}$ such that $v_i=c$ or $v_{i+1}=c$ for every $i$. Without loss of generality, $c=v_1$. Let $(v_i',v_{i+1}')$ be the directed edge in $\mathcal G$ corresponding to the node $(v_i,v_{i+1})$ in $\overline{L(\mathcal G)}$. Taking some edge in the cycle from $(v_i',v_{i+1}')$ to $(v_{i+2}',v_{i+3}')$, we have $v_{i+1}'=v_{i+2}'$. Further, we must have $c\in \{v_i',v_{i+1}'\}$ and $c\in \{v_{i+2}',v_{i+3}'\}$, so $c=v_{i+1}'=v_{i+2}'$.
\end{proof}

The following is the main result of the section, and describes the configurations that must occur in the $\Br(v,w)$ when there is no open $B_v$-orbit in $N^\vee_{X(w),e_v}C(v)$. 

\begin{theorem}\label{thm:openorbitclassification}
    Let $v\preceq w$. If $M^{v,w}(x)$ does not have full rank for every $x\in N^\vee_{X(w),e_v}C(v)$, then $\overline{\mcG_1(v,w)}$ has a cycle or there exist $\alpha_1,\alpha_2,\beta_1,\beta_2$ such that the configuration of \Cref{fig:flippedhighergjcycle} occurs in $\Br(v,w)$. 
\end{theorem}

\begin{figure}[h]
\begin{center}
    \includegraphics[width=8cm]{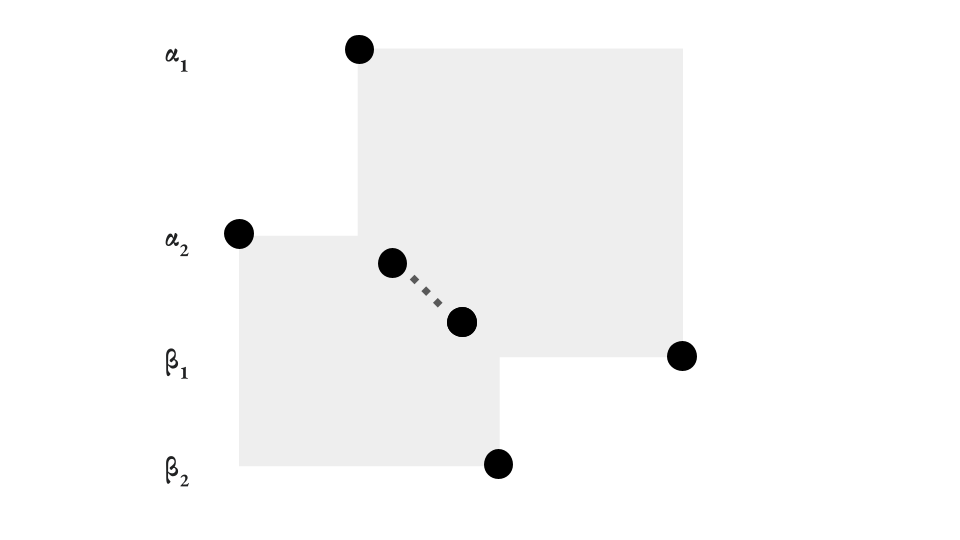}
\end{center}
\caption{On the shaded region, $d_{v,w}$ is positive, and the shaded region contains no filled dots other than the indicated ones. By the dotted line, we mean that the filled dots continue down and to the right in order.}
\label{fig:flippedhighergjcycle}
\end{figure}

\begin{proof}
    If every connected component of every $\overline{\mcG_i(v,w)}$ is a tree, we know that $M^{v,w}(\vec 1)$ has full rank by \Cref{prop:sufficientconditionforopen}. Thus it remains to show that $\alpha_1,\alpha_2,\beta_1,\beta_2$ in the configuration of \Cref{fig:flippedhighergjcycle} exist when $\overline{\mcG_1(v,w)}$ is a forest and some $\overline{\mcG_i(v,w)}$ contains a cycle. Take $j>1$ to be the smallest index such that $\overline{\mcG_j(v,w)}$ contains a cycle. The digraph $\mcG_j(v,w)$ is a subgraph of the line digraph of $\mcG_{j-1}(v,w)$, so by \Cref{lem:centralvertex}, there is some $(r,s)\in E_{j-2}(v,w)$ such that every vertex of the cycle is labeled by a vertex of the form $(a,s)\in E_{j-1}(v,w)$ or $(r,b)\in E_{j-1}(v,w)$. Further, every edge in the cycle is from some $(a,s)$ to some $(r,b)$. Choose the minimum such $a$, and call it $\alpha_1$. The two edges in the cycle containing $(\alpha_1,s)$ must be of the form $(\alpha_1,\beta_1)$ and $(\alpha_1,\beta_2)$. Without loss of generality, $\beta_1<\beta_2$. The other edge of the cycle containing $\beta_2$ must be of the form $(\alpha_2,\beta_2)$ for some $\alpha_2$. We have $\alpha_1<\alpha_2$ by the construction of $\alpha_1$, and we have $\alpha_2<r\leq s<\beta_1$. Also, both $\mcA_{\alpha_1\beta_1}(v)$ and $\mcA_{\alpha_2\beta_2}(v)$ are shaded in $\Br(v,w)$. 
    
    Because $\overline{\mcG_1(v,w)}$ is a forest, \Cref{lem:uniquepath} implies that $\Br(v,w)$ contains the configuration in \Cref{fig:flippedhighergjcycle}, with the central filled dots forming $(r,s)\in E_{j-2}(v,w)$. To see why, consider the following arguments.
    \begin{enumerate}
        \item The inequalities we have already established give us the correct relative positions of $\alpha_1$, $\alpha_2$, $\beta_1$, $\beta_2$, $r$, and $s$, as indicated in \Cref{fig:flippedhighergjcycle}.
        \item The shading is accurate, because the entire region is contained in $\mcA_{\alpha_1\beta_1}(v)$ and $\mcA_{\alpha_2\beta_2}(v)$, as can be seen in \Cref{fig:flippedhighergjcycle}.
        \item The filled dots continuing down and to the right in order is a consequence of \Cref{lem:uniquepath} and $(r,s)$ being in $E_{j-2}(v,w)$. Similarly, there are exactly $j-1$ filled dots in $(r,s)\times (v(r),v(s))$, and the dots are in the depicted order.
        \item There are no filled dots in the shaded region other than the ones indicated in \Cref{fig:flippedhighergjcycle}. This is because, in each of $\mcA_{\alpha_1\beta_1}(v)$ and $\mcA_{\alpha_2\beta_2}(v)$, there are already $(j-2)+2=j$ indicated dots; thus, having any more would contradict $(\alpha_1,\beta_1)\in E_j(v,w)$ or \newline $(\alpha_2,\beta_2)\in E_j(v,w)$. \qedhere
    \end{enumerate}
    
    % \begin{figure}[h]
    % \begin{center}
    %     \includegraphics[width=8cm]{Images/FlippedHigherGjCycleColored.png}
    % \end{center}
    % \caption{On the shaded region, $d_{v,w}$ is positive, and the shaded region contains no filled dots other than the indicated ones. By the dotted line, we mean that the filled dots continue down and to the right in order. The red rectangle is $\mcA_{\alpha_1\beta_1}$, and the blue rectangle is $\mcA_{\alpha_2\beta_2}(v)$.}
    % \label{fig:flippedhighergjcyclecolored}
    % \end{figure}
\end{proof}

\section{Proof of the Main Theorem}\label{sec:maintheorem}

In this section, we prove \Cref{thm:maintheorem}. In light of \Cref{thm:singularimpliesnoopenorbit} and \Cref{thm:openorbitclassification}, it remains to show that, if $X(w)$ is factorial and there is cycle in some $\overline{\mcG_i(v,w)}$, then $X(w_0v)$ is singular at $e_{w_0w}$. In \Cref{sec:preliminarylemmas}, we prove some preliminary lemmas, which in combination with \Cref{thm:openorbitclassification} relate singularity to open orbits. In \Cref{sec:proofoffinalimplication}, we complete the proof of this final implication, using a result of Dyer on the coefficient of $q$ in the Kazhdan--Lusztig polynomials $P_{v,w}(q)$.

\subsection{Preliminary Lemmas}\label{sec:preliminarylemmas}

Here, we will prove several auxiliary lemmas about Bruhat pictures which will be useful in our proof of \Cref{thm:maintheorem}.

\begin{lemma}\label{lem:lessthanshading}
    Let $v\preceq w$ and $v'\preceq w$. Then $v\preceq v'$ if and only if $d_{v',w}(p,q)\leq d_{v,w}(p,q)$ for all $p,q\in \mathbb Z$. 
\end{lemma}

\begin{proof}
    This is simply because, for all $p,q$,
    \begin{align*}
        r_w(p,q)-r_{v'}(p,q)\leq r_w(p,q)-r_{v}(p,q)
    \end{align*}
    is equivalent to $d_{v,v'}(p,q)\geq 0$, which is equivalent to $v\preceq v'$.
\end{proof}

The following corollary gives an interpretation of the singularity of $X(w)$ at $e_v$ in terms of $\Br(v,w)$. This ties the occurrence of \Cref{fig:flippedhighergjcycle} in \Cref{thm:openorbitclassification} to singularity.

\begin{corollary}\label{cor:singularcharacterization}
    Let $v\preceq w$. Then $X(w)$ is singular at $e_v$ if and only if there exist integers $\alpha_1,\dots\alpha_m$ and $\beta_1,\dots,\beta_k$ in $[n]$ such that the unfilled dots in these rows are in the configurations of \Cref{fig:nonmaximal4231}, \Cref{fig:nonmaximal3412}, or \Cref{fig:nonmaximal45312}. Further,
    \begin{enumerate}
        \item every indicated region is shaded;
        \item the shaded regions contain no unfilled dots, other than possibly in the rectangle \newline $(\alpha_1,\beta_2)\times (v(\alpha_2),v(\beta_1))$, as depicted in \Cref{fig:nonmaximal45312}; and
        \item in \Cref{fig:nonmaximal45312}, the unfilled dots in $(\alpha_1,\beta_2)\times (v(\alpha_2),v(\beta_1))$ must proceed down and to the left in order.
    \end{enumerate}
\end{corollary}

\begin{figure}[ht]
    \centering

    \begin{subfigure}[t]{0.32\textwidth}
        \centering
        \includegraphics[width=\linewidth]{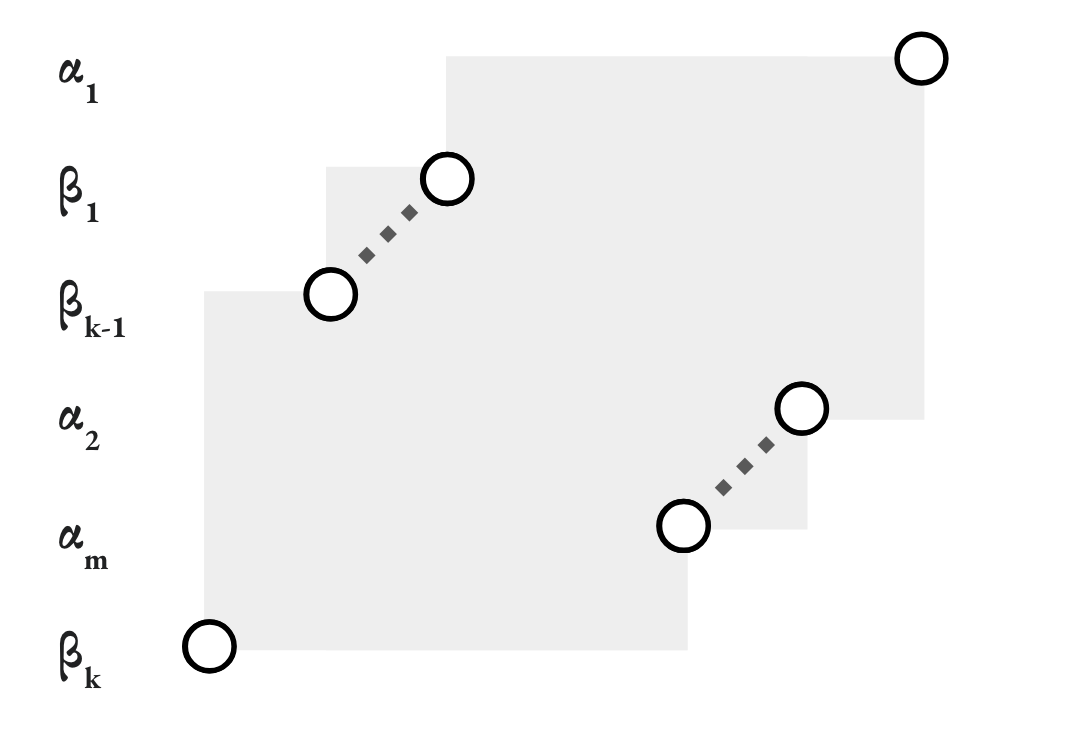}
        \caption{}
        \label{fig:nonmaximal4231}
    \end{subfigure}
    \hfill
    \begin{subfigure}[t]{0.34\textwidth}
        \centering
        \includegraphics[width=\linewidth]{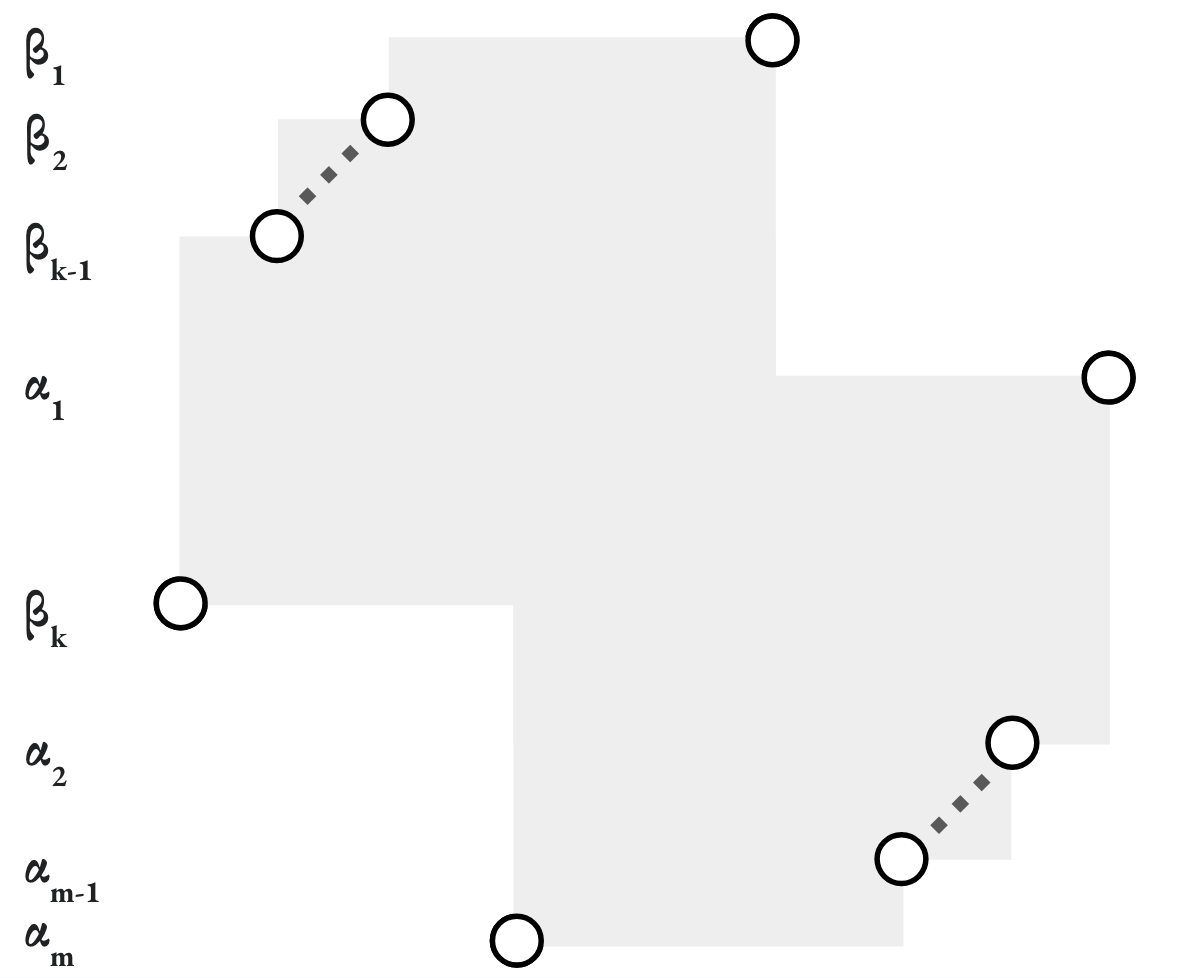}
        \caption{}
        \label{fig:nonmaximal3412}
    \end{subfigure}
    \hfill
    \begin{subfigure}[t]{0.32\textwidth}
        \centering
        \includegraphics[width=\linewidth]{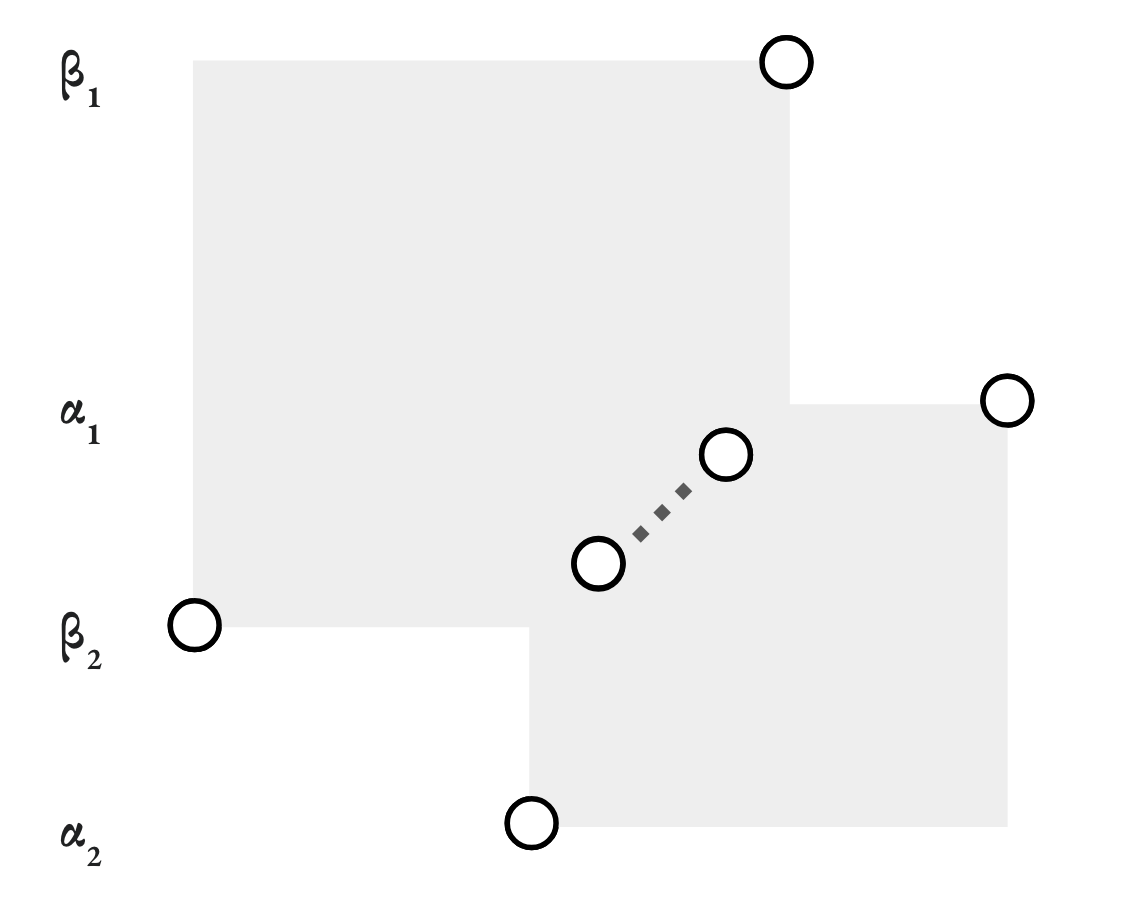}
        \caption{}
        \label{fig:nonmaximal45312}
    \end{subfigure}

    \caption{}
    \label{fig:nonmaximalconfigurations}
\end{figure}

\begin{proof}
    Note that $X(w)$ is singular at $e_v$ if and only if there exists $v'\succeq v$ such that $v'\in \maxsing(w)$. 
    
    First, assume that such a $v'$ exists. By \Cref{thm:maxsingclassification}, we have $v'=w\circ c$, with $c=(\alpha_1,\dots,\alpha_m,\beta_k,\dots,\beta_1)$, such that the pair $(v',w)$ satisfies the conditions of some case of \Cref{thm:maxsingclassification}. Then the unfilled dots in $\Br(v,w)$ will be in the same places as the unfilled dots in $\Br(v',w)$ by \Cref{rem:bruhatflipping}. Also, because $v\preceq v'$ and the indicated area is shaded in $\Br(v',w)$, then the indicated area is shaded in $\Br(v,w)$ by \Cref{lem:lessthanshading}.
    
    Conversely, assume that one of \Cref{fig:nonmaximal4231}, \Cref{fig:nonmaximal3412}, or \Cref{fig:nonmaximal45312} occurs in $\Br(v,w)$. Then letting $c=(\alpha_1,\dots,\alpha_m,\beta_k,\dots,\beta_1)$ and $v'=w\circ c$, we have $d_{v',w}\equiv 1$ on the shaded region of the relevant figure, and $d_{v',w}\equiv 0$ outside of that shaded region. By \Cref{lem:visuallydetermineR}, we have $d_{v,w}\geq 0$ outside of the shaded region, and by the definition of shading we have $d_{v,w}\geq 1$ on the shaded region. This gives $d_{v,w}\geq d_{v',w}$, and applying \Cref{lem:lessthanshading} yields $v\preceq v'$. Finally, $v'\in \maxsing(w)$ by \Cref{thm:maxsingclassification}, so $X(w)$ is singular at $e_v$. 
\end{proof}

\begin{lemma}\label{lem:fliplemma}
    Let $v\preceq w$. If $\mcA_{ij}(v)$ is shaded in $\Br(v,w)$, then $(i,j)\times (w_0v(j),w_0v(i))$ is shaded in $\Br(w_0w,w_0v)$.
\end{lemma}

\begin{proof}
     By definition, for any $p,q\in \mathbb Z$, we have
    \begin{align*}
        r_{w_0v}(p,w_0(q)+1)&=\#\{i\leq p:w_0v(i)\geq w_0(q)+1\}\\
        &=\#\{i\leq p:n-v(i)+1\geq n-q+2\}\\
        &=\#\{i\leq p:v(i)\leq q-1\}\\
        &=p-r_v(p,q),
    \end{align*}
    so
    \begin{align*}
        d_{w_0w,w_0v}(p,w_0(q)+1)&=r_{w_0v}(p,w_0(q)+1)-r_{w_0w}(p,w_0(q)+1)\\
        &=(p-r_v(p,q))-(p-r_w(p,q))\\
        &=d_{v,w}(p,q).
    \end{align*} 

    Now let $p\in (i,j)$ and $q\in (v(i),v(j)]$. Again by definition, we have $r_v(p,q)=r_{vs_{ij}}(p,q)-1$. Because $\mcA_{ij}(v)$ is shaded in $\Br(v,w)$, we have $v\prec vs_{ij}\preceq w$, so
    \begin{align*}
        d_{v,w}(p,q)&=r_w(p,q)-r_v(p,q)\\
        &=r_w(p,q)-r_{vs_{ij}}(p,q)+1\\
        &=d_{vs_{ij},w}(p,q)+1.
    \end{align*}
    By \Cref{lem:visuallydetermineR}, we have $d_{vs_{ij},w}(p',q')\geq 0$ for all $p',q'\in \mathbb Z$, so $d_{v,w}(p,q)\geq 1$.

    Finally, take $(p'',q'')$ in $(i,j)\times (w_0v(j),w_0v(i))$. Then $(p'',w_0(q'')+1)$ is in $(i,j)\times (v(i),v(j)]$, so
    \begin{align*}
        d_{w_0w,w_0v}(p'',q'')&=d_{v,w}(p'',w_0(q'')+1)\geq 1. \qedhere
    \end{align*}
\end{proof}

\Cref{cor:singularcharacterization} and \Cref{lem:fliplemma} give a characterization of the existence of an open $B_v$-orbit in $N_{X(w),e_v}^\vee C(v)$ in terms of the graphs $\mcG_i(v,w)$, which is of independent interest.

\begin{theorem}\label{thm:allforestcriterion}
     Let $v\preceq w$. Then $M^{v,w}(x)$ has full rank for some $x\in N^\vee_{X(w),e_v}C(v)$ if and only if every $\overline{\mcG_i(v,w)}$ is a forest. In fact, this equivalence gives a stronger version of \Cref{thm:matrixcharacterization}, in the sense that the existence of an open $B_v$-orbit in $N_{X(w),e_v}^\vee C(v)$ is equivalent to $M^{v,w}(\vec 1)$ having full rank.
\end{theorem}

\begin{proof}
    By \Cref{cor:nocycle}, if $\overline{\mcG_1(v,w)}$ contains a cycle, then $M^{v,w}(x)$ does not have full rank for any $x\in N^\vee_{X(w),e_v}C(v)$. If $\overline{\mcG_1(v,w)}$ is a forest and $\overline{\mcG_i(v,w)}$ for $i> 1$ contains a cycle, then $\Br(v,w)$ contains the configuration in \Cref{fig:flippedhighergjcycle} by the argument in the proof of \Cref{thm:openorbitclassification}. By \Cref{rem:bruhatflipping}, \Cref{lem:fliplemma}, and \Cref{cor:singularcharacterization}, $X(w_0v)$ is singular at $e_{w_0w}$, so $M^{v,w}(x)$ does not have full rank for any $x\in N^\vee_{X(w),e_v}C(v)$ by \Cref{thm:singularimpliesnoopenorbit}. Conversely, if every $\overline{\mcG_i(v,w)}$ is a forest, then $M^{v,w}(\vec 1)$ has full rank by \Cref{prop:sufficientconditionforopen}.

    For the second equivalence, \Cref{prop:sufficientconditionforopen} implies that if $M^{v,w}(\vec 1)$ does not have full rank, then there is some cycle in some $\overline{\mcG_i(v,w)}$, and we are done.
\end{proof}

\begin{remark}
    \Cref{thm:allforestcriterion} yields an $O(n^2)$ algorithm for determining whether there is an open $B$-orbit in $N^\vee_{X(w)}C(v)$. The steps are as follows.
    \begin{enumerate}
        \item Compute $d_{v,w}(p,q)$ for every $(p,q)\in [n]^2$.
        \item Compute $R(v,w)$ using the computed $d_{v,w}$. We do this in $O(n^2)$ by first computing the function
        \begin{align*}
            f(p,q)=\#\{(p',q')\in [1,p]\times [1,q]:d_{v,w}(p',q')=0\}
        \end{align*}
        for each $(p,q)\in [n]^2$, and then checking whether
        \begin{align*}
            \#\{(p,q)\in \mcA_{ij}(v):d_{v,w}(p,q)=0\}=&f(i,v(i))+f(j-1,v(j)-1)\\
            &-f(i,v(j)-1)-f(j-1,v(i))
        \end{align*}
        is $0$ for each pair $i<j$.
        \item Compute $E_1(v,w)$ and determine whether there is a cycle in $\overline{\mcG_1(v,w)}$. We do this in $O(n^2)$ because finding a cycle in a graph $\mcG=(V,E)$ can be done in $O(|V|+|E|)$.
        \item If $\overline{\mcG_1(v,w)}$ is a forest, sort the elements of $R(v,w)$ into each $E_i(v,w)$. We do this in $O(n^2)$ using \Cref{lem:uniquepath}.
        \item Check whether each $\overline{G_i(v,w)}$ contains a cycle. Again, this is $O(n^2)$ because finding a cycle in each takes $O(|E_i(v,w)|+|E_{i-1}(v,w)|)$.
    \end{enumerate}
\end{remark}

\subsection{Proof of \Cref{thm:maintheorem}}\label{sec:proofoffinalimplication}

In this subsection, we state a series of corollaries and equivalent formulations of results of Deodhar, Dyer, Tsukerman--Williams, and Bousquet--M\'elou--Butler. We then prove the main theorem \Cref{thm:maintheorem}.

Recall that there is a family of polynomials $P_{v,w}$ which was introduced by Kazhdan--Lusztig in \cite{KazhdanLusztig1979,KazhdanLusztig1980}, indexed by the pairs $v\preceq w$ of elements of $S_n$. Further, recall that $P_{v,w}=1$ if and only if $X(w)$ is smooth at $e_v$ by a result of Deodhar \cite{Deodhar1985}.

We establish the following straightforward consequence of a result of Dyer in \cite{Dyer1997}. 

\begin{theorem}\label{thm:dyer}
    For $v\preceq w$, we have
    \begin{align*}
        [q]P_{v,w}\geq |E_1(w_0w,w_0v)|-|E_1(v,w)|.
    \end{align*}
\end{theorem}

\begin{proof}
    In \cite{Dyer1997}, $|E_1(w_0w,w_0v)|$ is denoted by $c_{v,w}$, and $|E_1(v,w)|$ is denoted by $a_{v,w}$. The main theorem of \cite{Dyer1997} says that there is an integer $\Delta_{v,w}$ such that $|E_1(v,w)|-\Delta_{v,w}\geq 0$, and that
    \begin{align*}
        [q]P_{v,w}=|E_1(w_0w,w_0v)|-\Delta_{v,w}.
    \end{align*}
    Together, these yield the theorem.
\end{proof}

The following is a result of Bousquet--M\'elou--Butler, Theorem 1 of \cite{BousquetMelouButler2007}. More specifically, as is noted in the introduction of \cite{BousquetMelouButler2007}, the two theorems are equivalent by Proposition 2 of \cite{WooYongGorenstein}.

\begin{theorem}\label{thm:bousquetmeloubutler}
    Let $w\in S_n$. Then the following are equivalent.
    \begin{enumerate}
        \item $X(w)$ is factorial.
        \item The graph $\overline{\mcG_1(w_0w,w_0)}$ is a forest.
        \item $w$ avoids $4231$ and $3\underline{41}2$, where the underline means that the $4$ and $1$ must be adjacent in the one line notation for $w$. 
    \end{enumerate}
\end{theorem}

\begin{remark}
    Because smoothness of $X(w)$ is equivalent to $w$ avoiding $4231$ and $3412$, the first example of a permutation $w$ such that $X(w)$ is factorial but not smooth is $45312$ in $S_5$.
\end{remark}

The atom and coatom graphs, defined by Tsukerman--Williams in \cite{TsukermanWilliamsBruhatIntervalPolytopes}, are obtained from $\mcG_1(v,w)$ and $\mcG_1(w_0w,w_0v)$ by adding isolated vertices. Proposition 4.10 of \cite{TsukermanWilliamsBruhatIntervalPolytopes} says that the atom and coatom graphs have the same partition into connected components, and thus the same set of isolated vertices. Disregarding these vertices gives the following result.

\begin{theorem}\label{thm:tsukermanwilliams}
    For $v\preceq w$, the graphs $\overline{\mcG_1(v,w)}$ and $\overline{\mcG_1(w_0w,w_0v)}$ have the same number of connected components.
\end{theorem}

We now prove our main theorem, which is a rephrasing of \Cref{thm:maintheorem}.

\begin{theorem}\label{thm:mainthmproper}
    Let $v\preceq w$, and let $X(w)$ be factorial. Then there is an open $B_v$-orbit in $N_{X(w),e_v}^\vee C(v)$ if and only if $X(w_0v)$ is smooth at $e_{w_0w}$. 
\end{theorem}

\begin{proof}
    One implication is handled in \Cref{thm:singularimpliesnoopenorbit}. It is then left only to show that if there is no open $B_v$-orbit in $N_{X(w),e_v}^\vee C(v)$, then $X(w_0v)$ is singular at $e_{w_0w}$. By \Cref{thm:openorbitclassification}, if there is no open $B_v$-orbit in $N_{X(w),e_v}^\vee C(v)$, then $\overline{\mcG_1(v,w)}$ contains a cycle or $\overline{\mcG_1(v,w)}$ is a forest and $\Br(v,w)$ contains the configuration in \Cref{fig:flippedhighergjcycle}.

    \begin{enumerate}
        \item We begin with the case in which $\overline{\mcG_1(v,w)}$ is a forest and we have an occurrence of the configuration in \Cref{fig:flippedhighergjcycle} in $\Br(v,w)$. By \Cref{rem:bruhatflipping} and \Cref{lem:fliplemma}, applying $w_0$ to the pair $(v,w)$ preserves the shading, the lack of dots in the shaded region, and the relative positions of the dots, so we have the configuration in \Cref{fig:nonmaximal45312} in $\Br(w_0w,w_0v)$. Thus, by \Cref{cor:singularcharacterization}, $X(w_0v)$ is singular at $e_{w_0w}$.
        \item We now handle the case in which there is a cycle in $\overline{\mcG_1(v,w)}$. Because $X(w)$ is factorial, the graph $\overline{\mcG_1(w_0w,w_0)}$ is a forest by \Cref{thm:bousquetmeloubutler}. The graph $\overline{\mcG_1(w_0w,w_0v)}$ is a subgraph of $\overline{\mcG_1(w_0w,w_0)}$, so it is a forest as well. Therefore,
        \begin{align*}
            |E_1(w_0w,w_0v)|=\rank(A_{\mcG_1(w_0w,w_0v)})
        \end{align*}
        by \Cref{cor:allcomponentstree}. Similarly, because $\overline{\mcG_1(v,w)}$ contains a cycle, we have
        \begin{align*}
            |E_1(v,w)|>\rank(A_{\mcG_1(v,w)})
        \end{align*}
        by \Cref{cor:nocycle}. Combining \Cref{thm:tsukermanwilliams} and \Cref{thm:incidencematrixrank}, we have
        \begin{align*}
            \rank(A_{\mcG_1(v,w)})=\rank(A_{\mcG_1(w_0w,w_0v)}),
        \end{align*}
        so
        \begin{align*}
            |E_1(w_0w,w_0v)|&=\rank(A_{\mcG_1(v,w)})\\
            &=\rank(A_{\mcG_1(w_0w,w_0v)})\\
            &<|E_1(v,w)|.
        \end{align*}
        Applying \Cref{thm:dyer} to the pair $(w_0w,w_0v)$, this gives $[q]P_{w_0w,w_0v}>0$. For all $v\preceq w$, we have $P_{v,w}=1$ if and only if $X(w)$ is smooth at $e_v$ by the main theorem of \cite{Deodhar1985}, so $X(w_0v)$ is not smooth at $e_{w_0w}$. \qedhere
    \end{enumerate}
\end{proof}

\begin{remark}
    In fact, only the second case in the second implication in the proof of \Cref{thm:mainthmproper} requires the factoriality hypothesis.
\end{remark}

\section*{Acknowledgements}

This work was completed as part of the MIT PRIMES-USA program. The author thanks the anonymous PRIMES reviewer and Thomas R\" ud for helpful comments. The author also thanks Roman Bezrukavnikov for suggesting the main problem of the paper. Finally, the author would like to thank his PRIMES mentor Hamilton Wan for many helpful mathematical discussions, as well as substantial feedback which greatly affected the presentation and exposition of the paper.

This work did not receive any specific grant from funding agencies in the public, commercial, or not-for-profit sectors.

\section*{Declaration of Generative AI Use}

During the writing of this paper, the author used ChatGPT 5.6 Sol (OpenAI) for proofreading and copyediting. Additionally, ChatGPT suggested the use of \Cref{thm:dyer}, upon being asked whether there were existing results of the form ``$|E_1(w_0w,w_0v)|<|E_1(v,w)|$ implies $X(w_0v)$ is singular at $e_{w_0}$". Apart from this suggestion, all mathematical content and writing is the author's own.

\bibliographystyle{plain} 
\bibliography{cas-refs}

@article{billey2001maximalsingularlocischubert,
  author  = {Billey, Sara C. and Warrington, Gregory S.},
  title   = {Maximal singular loci of {Schubert} varieties in ${SL}(n)/B$},
  journal = {Transactions of the American Mathematical Society},
  volume  = {355},
  number  = {10},
  pages   = {3915--3945},
  year    = {2003},
  doi     = {10.1090/S0002-9947-03-03019-8},
  eprint  = {math/0102168},
  archivePrefix = {arXiv},
  primaryClass  = {math.AG}
}

@article{Lapid_2018,
  author  = {Lapid, Erez and M{\'i}nguez, Alberto},
  title   = {Geometric conditions for {$\square$}-irreducibility of certain representations of the general linear group over a non-archimedean local field},
  journal = {Advances in Mathematics},
  volume  = {339},
  pages   = {113--190},
  year    = {2018},
  doi     = {10.1016/j.aim.2018.09.027}
}

@article{LakshmibaiSeshadri1984,
  author  = {Lakshmibai, V. and Seshadri, C. S.},
  title   = {Singular locus of a {Schubert} variety},
  journal = {Bulletin of the American Mathematical Society},
  volume  = {11},
  number  = {2},
  pages   = {363--366},
  year    = {1984},
  doi     = {10.1090/S0273-0979-1984-15309-6}
}

@book{godsil01,
  author    = {Godsil, Chris and Royle, Gordon F.},
  title     = {Algebraic Graph Theory},
  series    = {Graduate Texts in Mathematics},
  volume    = {207},
  publisher = {Springer},
  address   = {New York},
  year      = {2001},
  doi       = {10.1007/978-1-4613-0163-9},
  isbn      = {978-1-4613-0163-9}
}

@article{BeilinsonBernstein1981,
  author  = {Beilinson, Alexander and Bernstein, Joseph},
  title   = {Localisation de $\mathfrak{g}$-modules},
  journal = {C. R. Acad. Sci. Paris S{\'e}r. I Math.},
  volume  = {292},
  pages   = {15--18},
  year    = {1981}
}

@article{BrylinskiKashiwara1981,
  author  = {Brylinski, Jean-Luc and Kashiwara, Masaki},
  title   = {{Kazhdan--Lusztig} conjecture and holonomic systems},
  journal = {Inventiones Mathematicae},
  volume  = {64},
  number  = {3},
  pages   = {387--410},
  year    = {1981},
  doi     = {10.1007/BF01389272}
}

@article{KazhdanLusztig1979,
  author  = {Kazhdan, David and Lusztig, George},
  title   = {Representations of {Coxeter} groups and {Hecke} algebras},
  journal = {Inventiones Mathematicae},
  volume  = {53},
  number  = {2},
  pages   = {165--184},
  year    = {1979},
  doi     = {10.1007/BF01390031}
}

@incollection{KazhdanLusztig1980,
  author    = {Kazhdan, David and Lusztig, George},
  title     = {{Schubert} varieties and {Poincar{\'e}} duality},
  booktitle = {Geometry of the Laplace Operator},
  series    = {Proceedings of Symposia in Pure Mathematics},
  volume    = {36},
  pages     = {185--203},
  publisher = {American Mathematical Society},
  address   = {Providence, RI},
  year      = {1980},
  doi       = {10.1090/pspum/036/573434}
}

@article{LakshmibaiSandhya1990,
  author  = {Lakshmibai, V. and Sandhya, B.},
  title   = {Criterion for smoothness of {Schubert} varieties in ${SL}(n)/B$},
  journal = {Proceedings of the Indian Academy of Sciences. Mathematical Sciences},
  volume  = {100},
  pages   = {45--52},
  year    = {1990},
  doi     = {10.1007/BF02881113}
}

@article{BousquetMelouButler2007,
  author  = {Bousquet-M{\'e}lou, Mireille and Butler, Steve},
  title   = {Forest-like permutations},
  journal = {Annals of Combinatorics},
  volume  = {11},
  number  = {3--4},
  pages   = {335--354},
  year    = {2007},
  doi     = {10.1007/s00026-007-0322-1}
}

@article{Manivel2001,
  author  = {Manivel, Laurent},
  title   = {Le lieu singulier des vari{\'e}t{\'e}s de {Schubert}},
  journal = {International Mathematics Research Notices},
  number  = {16},
  pages   = {849--871},
  year    = {2001},
  doi     = {10.1155/S1073792801000423}
}

@article{KasselLascouxReutenauer2003,
  author  = {Kassel, Christian and Lascoux, Alain and Reutenauer, Christophe},
  title   = {The singular locus of a {Schubert} variety},
  journal = {Journal of Algebra},
  volume  = {269},
  number  = {1},
  pages   = {74--108},
  year    = {2003},
  doi     = {10.1016/S0021-8693(03)00014-0}
}

@article{Cortez2003,
  author  = {Cortez, Aur{\'e}lie},
  title   = {Singularit{\'e}s g{\'e}n{\'e}riques et quasi-r{\'e}solutions des vari{\'e}t{\'e}s de {Schubert} pour le groupe lin{\'e}aire},
  journal = {Advances in Mathematics},
  volume  = {178},
  number  = {2},
  pages   = {396--445},
  year    = {2003},
  doi     = {10.1016/S0001-8708(02)00081-6}
}

@article{LapidMinguez2025,
  author  = {Lapid, Erez and M{\'i}nguez, Alberto},
  title   = {A binary operation on irreducible components of {Lusztig}'s nilpotent varieties {II}: applications and conjectures for representations of ${GL}_n$ over a non-archimedean local field},
  journal = {Pure and Applied Mathematics Quarterly},
  volume  = {21},
  number  = {2},
  pages   = {813--863},
  year    = {2025},
  doi     = {10.4310/PAMQ.241205005734}
}

@book{Harary1969GraphTheory,
  author    = {Harary, Frank},
  title     = {Graph Theory},
  publisher = {Addison-Wesley},
  address   = {Reading, MA},
  year      = {1969}
}

@article{Deodhar1985,
  author  = {Deodhar, Vinay V.},
  title   = {Local {Poincar{\'e}} duality and non-singularity of {Schubert} varieties},
  journal = {Communications in Algebra},
  volume  = {13},
  number  = {6},
  pages   = {1379--1388},
  year    = {1985},
  doi     = {10.1080/00927878508823227}
}

@incollection{Dyer1997,
  author    = {Dyer, Matthew J.},
  title     = {On coefficients of {$q$} in {Kazhdan--Lusztig} polynomials},
  booktitle = {Algebraic Groups and Lie Groups},
  editor    = {Lehrer, Gus I.},
  series    = {Australian Mathematical Society Lecture Series},
  volume    = {9},
  pages     = {189--194},
  publisher = {Cambridge University Press},
  address   = {Cambridge},
  year      = {1997}
}

@article{TsukermanWilliamsBruhatIntervalPolytopes,
  author  = {Tsukerman, Emmanuel and Williams, Lauren},
  title   = {{Bruhat} interval polytopes},
  journal = {Advances in Mathematics},
  volume  = {285},
  pages   = {766--810},
  year    = {2015},
  doi     = {10.1016/j.aim.2015.07.030},
  eprint  = {1406.5202},
  archivePrefix = {arXiv},
  primaryClass  = {math.CO}
}

@article{WooYongGorenstein,
  author  = {Woo, Alexander and Yong, Alexander},
  title   = {When is a {Schubert} variety {Gorenstein}?},
  journal = {Advances in Mathematics},
  volume  = {207},
  number  = {1},
  pages   = {205--220},
  year    = {2006},
  doi     = {10.1016/j.aim.2005.11.010}
}

@book{BjornerBrenti,
  author    = {Bj{\"o}rner, Anders and Brenti, Francesco},
  title     = {Combinatorics of {Coxeter} Groups},
  series    = {Graduate Texts in Mathematics},
  volume    = {231},
  publisher = {Springer},
  address   = {New York},
  year      = {2005},
  doi       = {10.1007/3-540-27596-7}
}

@article{GeissLeclercSchroer2011,
  author  = {Christof Geiss and Bernard Leclerc and Jan Schr{\"o}er},
  title   = {Kac--Moody Groups and Cluster Algebras},
  journal = {Advances in Mathematics},
  volume  = {228},
  number  = {1},
  pages   = {329--433},
  year    = {2011},
  doi     = {10.1016/j.aim.2011.05.011}
}

@article{GeissLeclercSchroer2005,
  author  = {Christof Gei{\ss} and Bernard Leclerc and Jan Schr{\"o}er},
  title   = {Semicanonical Bases and Preprojective Algebras},
  journal = {Annales Scientifiques de l'{\'E}cole Normale Sup{\'e}rieure},
  volume  = {38},
  number  = {2},
  pages   = {193--253},
  year    = {2005},
  doi     = {10.1016/j.ansens.2004.12.001}
}

@article{GeissLeclercSchroer2006,
  author  = {Christof Gei{\ss} and Bernard Leclerc and Jan Schr{\"o}er},
  title   = {Rigid Modules over Preprojective Algebras},
  journal = {Inventiones Mathematicae},
  volume  = {165},
  number  = {3},
  pages   = {589--632},
  year    = {2006},
  doi     = {10.1007/s00222-006-0507-y}
}

@article{AizenbudLapid2025,
  author  = {Avraham Aizenbud and Erez Lapid},
  title   = {A Binary Operation on Irreducible Components of Lusztig's Nilpotent Varieties {I}: Definition and Properties},
  journal = {Pure and Applied Mathematics Quarterly},
  volume  = {21},
  number  = {1},
  pages   = {5--41},
  year    = {2025},
  doi     = {10.4310/PAMQ.241203024730}
}

@misc{Lapid2026,
  author       = {Erez Lapid},
  title        = {A greedy open-orbit criterion for solvable algebraic group actions, with applications to Lusztig's nilpotent varieties},
  year         = {2026},
  eprint       = {2608.14413},
  archivePrefix = {arXiv},
  primaryClass = {math.RT},
  doi          = {10.48550/arXiv.2608.14413}
}

@article{LapidMinguez2020,
  author  = {Erez Lapid and Alberto M{\'i}nguez},
  title   = {Conjectures and Results about Parabolic Induction of Representations of {$GL_n(F)$}},
  journal = {Inventiones Mathematicae},
  volume  = {222},
  number  = {3},
  pages   = {695--747},
  year    = {2020},
  doi     = {10.1007/s00222-020-00982-7}
}

@article{GeissSchroer2005,
  author  = {Christof Gei{\ss} and Jan Schr{\"o}er},
  title   = {Extension-Orthogonal Components of Preprojective Varieties},
  journal = {Transactions of the American Mathematical Society},
  volume  = {357},
  number  = {5},
  pages   = {1953--1962},
  year    = {2005},
  doi     = {10.1090/S0002-9947-04-03555-X}
}

@article{JensenSuYu2009,
  author  = {Bernt Tore Jensen and Xiuping Su and Rupert Wei Tze Yu},
  title   = {Rigid Representations of a Double Quiver of Type {$A$}, and Richardson Elements in Seaweed Lie Algebras},
  journal = {Bulletin of the London Mathematical Society},
  volume  = {41},
  number  = {1},
  pages   = {1--15},
  year    = {2009},
  doi     = {10.1112/blms/bdn087}
}



\end{document}